\documentclass{amsart}

\usepackage{amssymb, amsmath, color}
\usepackage[a4paper,top=3cm,bottom=3cm,left=1.5 cm,right=1.5cm]{geometry}

\usepackage{amsaddr}

\usepackage[T1]{fontenc}                                     
\usepackage[utf8]{inputenc}                                 
\usepackage[english]{babel}                                   

\usepackage{multicol}

\usepackage{amsmath}                                         
\usepackage{amssymb}                                         
\usepackage{graphicx}                                        
\usepackage{booktabs}                                         
\usepackage{mathtools}  
\usepackage{braket}

\usepackage{url}
\usepackage{appendix}
\usepackage{quoting}
\quotingsetup{font=small}

\usepackage{amsthm}                                         
\theoremstyle{plain}

\theoremstyle{definition}

\theoremstyle{plain}
\newtheorem{theorem}{Theorem}[section]

\newtheorem{proposition}{Proposition}[section]
\newtheorem{remark}{Remark}[section]
\newtheorem{conjecture}{Conjecture}[section]
\newtheorem{example}{Example}[section]

\usepackage{indentfirst}

\usepackage{amssymb, amsmath}
\usepackage[utf8]{}
\usepackage[english]{babel}
\title{On the identities and cocharacters of the algebra of 3 $\times$ 3 matrices with orthosymplectic superinvolution}
\date{}
\author{Sara Accomando}
\address{Dipartimento di Matematica e Informatica, Università degli Studi di Palermo, Via Archirafi, 34, 90123, Palermo, Italy. 
\\{\em E-mail}: {\tt sara.accomando@unipa.it}}

\begin{document}
\renewcommand{\abstractname}{\textbf{Abstract}}

\begin{abstract}
Let $M_{1,2}(F)$ be the algebra of $3 \times 3$ matrices with orthosymplectic superinvolution $*$ over a field $F$ of characteristic zero. We study the $*$-identities of this algebra through the representation theory of the group $\mathbb{H}_n = (\mathbb{Z}_2 \times \mathbb{Z}_2) \sim S_n$. We decompose the space of multilinear $*$-identities of degree $n$ into the sum of irreducibles under the $\mathbb{H}_n$-action in order to study the irreducible characters appearing in this decomposition with non-zero multiplicity. Moreover, by using the representation theory of the general linear group, we determine all the $*$-polynomial identities of $M_{1,2}(F)$ up to degree $3$.

\

\end{abstract}

\maketitle

\section{Introduction}

Let $F$ be a field of characteristic zero and let $F \langle X \rangle$ be the free associative algebra on a countable set $X = \{ x_1, x_2, \dots \}$ over $F$.
An associative $F$-algebra $A$ is a $PI$-algebra if it satisfies a non-trivial polynomial $f(x_1, \dots , x_n) \in F \langle X \rangle$, i.e., if $f(a_1, \dots , a_n) = 0$, for all $a_1, \dots , a_n \in A$. We say that $f \equiv 0$ is a polynomial identity for the algebra $A$. The set of all the polynomial identities satisfied by $A$, $Id(A)$, is a $T$-ideal of the free algebra, i.e., it is an ideal invariant under all the endomorphisms of $F \langle X \rangle$.

In case of characteristic zero, Specht \cite{Sp} conjectured that every $T$-ideal of the free algebra is finitely generated and it was proved in the affirmative by Kemer in 1987 \cite{K}. However the generators are known only in very few cases and the problem is still open for $M_k(F)$, with $k \geq 3$, where $M_k(F)$ is the algebra of $k \times k$ matrices over $F$. 
For this reason, since in characteristic zero every $T$-ideal is generated by multilinear polynomials, one studies $P_n \cap Id(A)$, $n \geq 1$, where $P_n$ is the space of multilinear polynomials of degree $n$.
To this end, we consider the representation theory of the symmetric group $S_n$ and we define an action on $P_n$: $\sigma f(x_1, \dots , x_n) := f(x_{\sigma(1)}, \dots , x_{\sigma(n)})$. The space $P_n \cap Id(A)$ is invariant under this action and, so, $\frac{P_n}{P_n \cap Id(A)}$ has a structure of $S_n$-module and its character is called the $n$-th cocharacter of $A$.
In \cite{B}, Benanti determined the conditions such that the multiplicities are non-zero in the decomposition of the $n$-th cocharacter of $M_3(F)$ in irreducible $S_n$-character.

In this paper, we are interested in studying the polynomial identities satisfied by $M_3(F)$ with an additional structure. This study was started in \cite{BC, LM}.
Here we consider the superalgebra $M_3(F)$ endowed with a superinvolution. 
Many researchers dealt with this additional structure (see \cite{EV, R}). In the theory of polynomial identities, superalgebras with superinvolutions are of particular interest, since in \cite{AGK}, Aljadeff, Giambruno and Karasik proved that there exists a strict connection between algebras with involution and finite dimensional superalgebras with superinvolution.

During the years, researchers obtained some information about the $*$-identities satisfied by a superalgebra with superinvolution, through the study of the sequences attached to a PI-algebra (see \cite{CISV, GILM1, GILM2, I, ILM, IM}).

In particular, the superinvolutions on $M_n(F)$ were classified by Racine in \cite{R}. He proved that there are two types, the transpose and the orthosymplectic superinvolution. The study of the $*$-identities satisfied by these superalgebras with superinvolutions was started by Giambruno, Ioppolo and Martino in \cite{GIM}. They focused on the standard polynomials and they determined the minimal degree of a standard polynomial vanishing on suitable subsets of symmetric or skew matrices for both types of superinvolutions.

A complete characterization of the cocharacter of $M_3(F)$ with transpose superinvolution was done in \cite{CSV}, in which they also determine all the $*$-identities up to degree $3$.

Here we consider $M_3(F)$ with orthosymplectic superinvolution $*$ and we study the $*$-identities through the representation theory of the group $\mathbb{H}_n = (\mathbb{Z}_2 \times \mathbb{Z}_2) \sim S_n$. We decompose the space of multilinear $*$-identities of degree $n$ into the sum of irreducibles under the $\mathbb{H}_n$-action in order to study the irreducible characters appearing in this decomposition with non-zero multiplicity. Moreover, by using the representation theory of the general linear group, we determine all the $*$-polynomial identities of $M_{1,2}(F)$ up to degree $3$.

\section{Preliminaries}

Throughout the paper, we will denote by $F$ a field of characteristic zero. An associative algebra $A$ is a \textit{$\mathbb{Z}_2$-graded algebra} or a \textit{superalgebra} if it has a vector space decomposition $A = A_0 \oplus A_1$ such that $A_0 A_0 + A_1 A_1 \subseteq A_0$ and $A_0 A_1 + A_1 A_0 \subseteq A_1$.
The elements of $A_0$ and $A_1$ are called \textit{homogeneous of degree zero} (or even degree) and of \textit{degree one} (or odd degree), respectively.

Let $A = A_0 \oplus A_1$ be a superalgebra. A superinvolution on $A$ is a graded linear map $*: A \longrightarrow A$, i.e., a map preserving the grading, such that $(a^{*})^{*} = a$ and $(ab)^{*}=(-1)^{|a||b|}b^{*}a^{*}$ where $|c|$ denotes the homogeneous degree of the element $c \in A$.
A superalgebra endowed with a superinvolution is called a \textit{${*}$-superalgebra}. Since char$F = 0$, the $*$-superalgebra can be written as
   $$A = A_0^+ \oplus A_0^- \oplus A_1^+ \oplus A_1^-,$$
where, for $i = 0, 1$, $A_i^+ = \{a \in A_i \ | \ a^{*} = a\}$ and $A_i^- = \{a \in A \ | \ a^{*} = -a\}$ denote the sets of homogeneous symmetric and skew elements of $A_i$, respectively.

Let $X = \{ x_1, x_2, \dots \}$ be a countable set of non-commuting variables. We write $X = Y \cup Z$ as disjoint union of two infinite homogeneous subsets $Y = \{ y_1, y_2, \dots \}$ and $Z = \{ z_1, z_2, \dots \}$ of degree $0$ and $1$, respectively. Then $F \langle Y \cup Z \rangle = \langle y_1, z_1, y_2, z_2, \dots \rangle$ has a natural structure of free superalgebra if we require that the variables from $Y$ have degree zero and the variables from $Z$ have degree one. This algebra is called the free superalgebra over $F$. 
Moreover, if we write each set as the disjoint union of two infinite sets of symmetric and skew elements, respectively, then we have the free $*$-superalgebra
   $$F \langle Y \cup Z, * \rangle = F \langle y_1^+, y_1^-, z_1^+, z_1^-, y_2^+, y_2^-, z_2^+, z_2^-, \dots \rangle,$$
where $y_i^+ = y_i + y_i^*$ denotes a symmetric variable of even degree, $y_i^- = y_i - y_i^*$ a skew symmetric variable of even degree, $z_i^+ = z_i + z_i^*$ a symmetric variable of odd degree and $z_i^- = z_i - z_i^*$ a skew variables of odd degree.

We say that $f(y_1^+, \dots, y_n^+, y_1^-, \dots, y_m^-, z_1^+, \dots, z_t^+, z_1^-, \dots, z_s^-) \in F \langle Y \cup Z, * \rangle$ is a $*$-identity of $A$, and we write $f \equiv 0$, if $f(u_1^+, \dots, u_n^+, u_1^-, \dots, u_m^-, v_1^+, \dots, v_t^+, v_1^-, \dots, v_s^-) = 0,$ for all $u_1^+, \dots, u_n^+ \in A_0^+$, $u_1^-, \dots, u_m^- \in A_0^-$, $v_1^+, \dots, v_t^+ \in A_1^+$, $v_1^-, \dots, v_s^- \in A_1^-$.

We consider the set of all $*$-identities of $A$ 
   $$Id_2^*(A) = \{f \in F \langle Y \cup Z, * \rangle : f \equiv 0 \text{ on } A\}$$
which is a $T_2^{*}$-ideal of $F \langle Y \cup Z, * \rangle$, i.e., an ideal invariant under all graded endomorphisms of the free superalgebra  commuting with the superinvolution $*$.

It is well known that in characteristic zero every $*$-identity is equivalent to a system of multilinear $*$-identities.
We denote by $$P_n^{*} = span_F\{w_{\sigma (1)} \dots w_{\sigma (n)} \ | \ \sigma \in S_n, \ w_i \in \{y_i^+, y_i^-, z_i^+, z_i^-\}, \ i = 1, \dots , n\}$$ the space of all multilinear $*$-polynomials of degree $n$ in $y_1^+, y_1^-, z_1^+, z_1^-, \dots , y_n^+, y_n^-, z_n^+, z_n^-$.
Then, the study of $Id_2^{*}(A)$ is equivalent to the study of $P_n^{*} \cap Id_2^{*}(A), \ \forall \ n \geq 1$.

Now, we consider the group $$\mathbb{H}_n = (\mathbb{Z}_2 \times \mathbb{Z}_2) \sim S_n = \{((g_1, h_1), \dots, (g_n, h_n); \ \sigma) : (g_i, h_i) \in (\mathbb{Z}_2 \times \mathbb{Z}_2), \ \sigma \in S_n\},$$ with multiplication given by $$((g_1, h_1), \dots, (g_n, h_n); \ \sigma)((a_1, b_1), \dots, (a_n, b_n); \ \tau) = ((\bar g_1, \bar h_1), \dots, (\bar g_n, \bar h_n); \ \sigma \tau),$$ where $\bar g_i = g_i a_{\sigma^{-1}(i)}$ and $\bar h_i = h_i b_{\sigma^{-1}(i)}$, for all $1 \leq i \leq n$.

If we write $\mathbb{Z}_2 \times \mathbb{Z}_2 = \{1, *\} \times \{1, \zeta\} = \{1, *, \zeta, *\zeta \}$, we get $\mathbb{H}_n$ acting on the left on $P_n^{*}$ by setting, for any $h = (a_1, \dots , a_n; \sigma) \in \mathbb{H}_n$:
\begin{center}
  $h y_i^+ = y_{\sigma(i)}^+, \hspace{5,6 cm} h y_i^- = \begin{cases} y_{\sigma(i)}^-, & \text{if $a_{\sigma(i)} \in \{1, \zeta \}$} \\ -y_{\sigma(i)}^-, & \text{if $a_{\sigma(i)} \in \{*, * \zeta \}$}\end{cases},$
 \end{center}
 \begin{center}
  $h z_i^+ = \begin{cases} z_{\sigma(i)}^+, & \text{if $a_{\sigma(i)} \in \{1, * \}$} \\ -z_{\sigma(i)}^+, & \text{if $a_{\sigma(i)} \in \{\zeta, * \zeta \}$}\end{cases}, \hspace{2 cm} h z_i^- = \begin{cases} z_{\sigma(i)}^-, & \text{if $a_{\sigma(i)} \in \{1, * \zeta \}$} \\ -z_{\sigma(i)}^-, & \text{if $a_{\sigma(i)} \in \{*, \zeta \}$}\end{cases},$
 \end{center}
for any $i = 1, \dots, n$. Since $P_n^{*} \cap Id_2^{*}(A)$ is invariant under this action, 
 $$P_n^{*}(A) = \frac{P_n^{*}}{P_n^{*} \cap Id_2^{*}(A)}$$
has a structure of $\mathbb{H}_n$-module, for all $n \geq 1$. Its character, $\chi_n^{*}(A)$, is called the \textit{$n$th $*$-cocharacter of $A$} and the sequence $\{\chi_n^{*}(A)\}_{n \geq 1}$ is the \textit{${*}$-cocharacter sequence of $A$}. 

Let $n \geq 1$ and write $n = n_1 + n_2 + n_3 + n_4$. We define $\langle n \rangle = (n_1, n_2, n_3, n_4)$, a composition of $n$, as a sum of four non-negative integers. We say that $\langle \lambda \rangle$ is a multipartition of $n = n_1 + n_2 + n_3 + n_4$, and we write $\langle \lambda \rangle \vdash \langle n \rangle$ (or $\langle \lambda \rangle \vdash n$), if $\langle \lambda \rangle = (\lambda(1), \lambda(2), \lambda(3), \lambda(4))$, with $\lambda(i) \vdash n_i, \ i = 1, \dots, 4$.

Since char$F = 0$, there is a one-to-one correspondence between the irreducible $\mathbb{H}_n$-characters and the multipartitions $\langle \lambda \rangle \vdash n$. More precisely,
     \begin{equation}\label{eqn:cocarattere1}
         \begin{split}
             \chi_n^{*}(A) = \sum_{\langle \lambda \rangle \vdash n} m_{\langle \lambda \rangle} \chi_ {\langle \lambda \rangle},
         \end{split}
     \end{equation}
where $\chi_{\langle \lambda \rangle}$ is the irreducible $\mathbb{H}_n$-character associated to the multipartition $\langle \lambda \rangle$ with corresponding multiplicity $m_{\langle \lambda \rangle} \geq 0$. 

For $\langle n \rangle = (n_1, n_2, n_3, n_4)$ fixed, let $P_{n_1, n_2, n_3, n_4} \subseteq P_n^{*}$ be the vector space of the multilinear $*$-polynomials in which the first $n_1$ variables are symmetric of degree zero, the next $n_2$ variables are skew of degree zero, the next $n_3$ variables are symmetric of degree one and the last $n_4$ variables are skew of degree one. The group $S_{n_1} \times S_{n_2} \times S_{n_3} \times S_{n_4}$ acts on the left on the vector space $P_{n_1, n_2, n_3, n_4}$ by permuting the variables of the same homogeneous degree which are all symmetric or all skew at the same time. So $S_{n_1}$ permutes the even symmetric variables, $S_{n_2}$ permutes the even skew variables, $S_{n_3}$ permutes the odd symmetric variables and $S_{n_4}$ permutes the odd skew variables. In this way, $P_{n_1, n_2, n_3, n_4}$ becomes an $S_{n_1} \times S_{n_2} \times S_{n_3} \times S_{n_4}$-module. Since $P_{n_1, n_2, n_3, n_4} \cap Id^*(A)$ is invariant under this action, we get that 
 $$P_{n_1, n_2, n_3, n_4}(A) = \frac{P_{n_1, n_2, n_3, n_4}}{P_{n_1, n_2, n_3, n_4} \cap Id_2^{*}(A)}$$
has a induced structure of $S_{n_1} \times S_{n_2} \times S_{n_3} \times S_{n_4}$-module. We denote by $\chi_{n_1, n_2, n_3, n_4}(A)$ its character, which is called the \textit{$(n_1, n_2, n_3, n_4)$-th cocharacter} of $A$.

Since char$F = 0$, then by complete reducibility, we write it as a sum of irreducible characters:
  \begin{equation}\label{eqn:cocarattere4}
          \chi_{n_1, n_2, n_3, n_4}(A) = \sum_{\langle \lambda \rangle \vdash n} \bar{m}_{\langle \lambda \rangle} \chi_{\lambda(1)} \otimes \cdots \otimes \chi_{\lambda(4)},
  \end{equation}
where $\bar{m}_{\langle \lambda \rangle} \geq 0$ is the multiplicity of $\chi_{\lambda(1)} \otimes \cdots \otimes \chi_{\lambda(4)}$ in $\chi_{n_1, n_2, n_3, n_4}(A)$.

By a generalization of \cite[Theorem 1.3]{DG}, we obtain the following.

\begin{theorem}
    In the decompositions given in \eqref{eqn:cocarattere1} and \eqref{eqn:cocarattere4}, $m_{\langle \lambda \rangle} = \bar{m}_{\langle \lambda \rangle}$, for all $\langle \lambda \rangle \vdash n$.
\end{theorem}

In particular, if $A$ is a finite dimensional algebra with dim$A_0^+ = d_1$, dim$A_0^- = d_2$, dim$A_1^+ = d_3$, dim$A_1^- = d_4$, 
   \begin{equation}\label{eqn:cocarattere2}
       \begin{split}
            \chi_n^{*}(A) = \sum_{\substack{\langle \lambda \rangle \vdash n, \\ h(\lambda(1)) \leq d_1, \ h(\lambda(2)) \leq d_2, \\ h(\lambda(3)) \leq d_3, \ h(\lambda(4)) \leq d_4}} m_{\langle \lambda \rangle} \chi_ {\langle \lambda \rangle},
       \end{split}
   \end{equation}    
where, for $i = 1, \dots, 4$, $h(\lambda(i))$ denotes the height of the partition $\lambda(i)$ (see \cite[Lemma 1.2]{DG}).

Now, for $m \geq 1$, let $F_m = F_m \langle Y \cup Z, * \rangle$ be the space of $*$-polynomials in the variables $y_1^+, \dots, y_m^+$, $y_1^-, \dots, y_m^-$, $z_1^+, \dots, z_m^+$, $z_1^-, \dots, z_m^-$. Let $V_1 = span_F\{y_1^+, \dots , y_m^+\}$, $V_2 = span_F\{y_1^-, \dots , y_m^-\}$, $V_3 = span_F\{z_1^+, \dots , z_m^+\}$ and $V_4 = span_F\{z_1^-, \dots , z_m^-\}$. Then the group $GL(V_1) \times GL(V_2) \times GL(V_3) \times GL(V_4) \cong GL_m^4 = GL_m \times GL_m \times GL_m \times GL_m$ acts naturally on the left on $V_1 \oplus V_2 \oplus V_3 \oplus V_4$ and this action can be diagonally extended to an action on $F_m$. Here $GL(V_i)$ is the group of all the automorphisms of the vector space $V_i$ and $GL_m$ denotes the general linear group of degree $m$.

Consider $F_m^n$ the subspace of all homogeneous $*$-polynomials of $F_m$ of degree $n \geq m$ which is a $GL_m^4$-submodule of $F_m$. Since $F_m^n$ is a $GL_m^4$-module and $F_m^n \cap Id_2^{*}(A)$ is invariant under this action, 
 $$F_m^n(A) = \frac{F_m^n}{F_m^n \cap Id_2^{*}(A)}$$ 
has a structure of $GL_m^4$-module and we denote by $\psi_n^{*}(A)$ its character.

There exists a one-to-one correspondence between the irreducible $GL_m^4$-characters and the multipartitions $\langle \lambda \rangle = (\lambda(1), \dots , \lambda(4))$ of $n$, where the $\lambda(i)$'s are partitions with at most $m$ parts, $i = 1, \dots, 4$ (see \cite[Theorem 12.4.4]{Dr}).

Hence denoted by $\psi_{\langle \lambda \rangle}$ the irreducible $GL_m^4$-character corresponding to the multipartition $\langle \lambda \rangle$, we have:
 \begin{equation}\label{eqn:cocarattere3}
     \psi_n^{*}(A) = \sum_{\substack{\langle \lambda \rangle \vdash n, \\ h(\langle \lambda \rangle) \leq m}} \bar{\bar{m}}_{\langle \lambda \rangle} \psi_{\langle \lambda \rangle},
 \end{equation}
where $\bar{\bar{m}}_{\langle \lambda \rangle} \geq 0$ and $h(\langle \lambda \rangle) = max \{h(\lambda(i)), \ i = 1, \dots, 4\}$.

For an extension of the result in the involution case (see \cite[Theorem 3]{G}), we get the following.

\begin{theorem}\label{th:cocharacter}
In the decompositions given in \eqref{eqn:cocarattere1} and \eqref{eqn:cocarattere3}, $m_{\langle \lambda \rangle} = \bar{\bar{m}}_{\langle \lambda \rangle}$, for all $\langle \lambda \rangle \vdash n$ such that $h(\langle \lambda \rangle) \leq m$.
\end{theorem}

We recall that an irreducible $GL_m^4$-submodule $W^{\langle \lambda \rangle}$ of $F_m^n(A)$ is generated by a non-zero $*$-polynomial $f_{\langle \lambda \rangle}$ called \textit{highest weight vector} associated to the multipartition $\langle \lambda \rangle$ (see \cite[Theorem 12.4.12]{Dr}).

A multitableau $T_{\langle \lambda \rangle} = (T_{\lambda(1)}, T_{\lambda(2)}, T_{\lambda(3)}, T_{\lambda(4)})$ is a 4-tuple of Young tableaux $T_{\lambda(i)}, \ 1 \leq i \leq 4$.
The multitableau $\bar T_{\langle \lambda \rangle} = (\bar T_{\lambda(1)}, \dots , \bar T_{\lambda(4)})$ such that $1, \dots , n$ are inserted, in this order, from top to bottom, from left to right, column by column, from the tableau $\bar T_{\lambda(1)}$ to the tableau $\bar T_{\lambda(4)}$ is called initial multitableau of shape $\langle \lambda \rangle$. The initial multitableau is a standard multitableau, that is, each $\bar T_{\lambda (i)}, \ i =1, \dots, 4$ is a standard Young tableau.
The highest weight vector associated is called initial highest weight vector and it is given by

    \begin{equation}\label{eqn:highestweightvector}
        f_{\bar T_{\langle \lambda \rangle}} = \prod_{i = 1}^{\lambda(1)_1} St_{h_i(\lambda(1))}(y_1^+, \dots, y_{h_i(\lambda(1))}^+) \prod_{i = 1}^{\lambda(2)_1} St_{h_i(\lambda(2))}(y_1^-, \dots, y_{h_i(\lambda(2))}^-) 
    \end{equation}    
         $$\prod_{i = 1}^{\lambda(3)_1} St_{h_i(\lambda(3))}(z_1^+, \dots, z_{h_i(\lambda(3))}^+) \prod_{i = 1}^{\lambda(4)_1} St_{h_i(\lambda(4))}(z_1^-, \dots, z_{h_i(\lambda(4))}^-),$$

\noindent where $h_i(\lambda(j))$ is the height of the $i$th column of the Young diagram corresponding to the partition $\lambda(j)$, $\lambda(j)_1$ is the first element of the partition $\lambda(j)$, for all $j = 1, \dots, 4$, and $St_r (x_1, \dots, x_r) = \sum_{\theta \in S_r} sgn(\theta) x_{\theta(1)} \dots x_{\theta(r)}$ is the \textit{standard polynomial of degree $r$}.

For a fixed multitableau $T_{\langle \lambda \rangle}$ we denote by $f_{T_{\langle \lambda \rangle}} = f_{\bar T_{\langle \lambda \rangle}} \sigma^{-1}$ the highest weight vector associated to $T_{\langle \lambda \rangle}$, where $\sigma$ is the only element of $S_n$ transforming $\bar T_{\langle \lambda \rangle}$ in $T_{\langle \lambda \rangle}$ and $S_n$ acts on the right on $F_m^n$ by permuting places in which the variables occur.

From \cite[Proposition 15]{LM}, we have the following

\begin{proposition}\label{prop:hwv}
For all $\langle \lambda \rangle \vdash n$, $f_{\langle \lambda \rangle}$ can be expressed uniquely as a linear combination of vectors $f_{T_{\langle \lambda \rangle}}$, where $T_{\langle \lambda \rangle}$ is a standard multitableau.
\end{proposition}

\begin{theorem}\label{th:hwv}\textnormal{\cite[Theorem 12.4.4]{Dr}}
In the decomposition \eqref{eqn:cocarattere3}, $\bar{\bar{m}}_{\langle \lambda \rangle} \not = 0$ if and only if there exists a multitableau $T_{\langle \lambda \rangle}$ such that $f_{T_{\langle \lambda \rangle}} \notin Id_2^{*}(A)$. Moreover, $\bar{\bar{m}}_{\langle \lambda \rangle}$ is equal to the maximal number of highest weight vectors $f_{T_{\langle \lambda \rangle}} \notin Id_2^{*}(A)$ which are linearly independent in $F_m^n(A)$.
\end{theorem}

\section{Algebra of 3 $\times$ 3 matrices with orthosymplectic superinvolution}

Let $M_n(F)$ be the algebra of $n \times n$ matrices over a field $F$ of characteristic zero. It is well known that, up to isomorphism, a $\mathbb{Z}_2$-grading on $M_n(F)$ is given by
     $$M_{k, h}(F) := \left \{ \left 
      (\begin{array}{cc}
         X & 0 \\
         0 & T
       \end{array} \right ) \right \} \oplus 
      \left \{ \left (\begin{array}{cc}
         0 & Y \\
         Z & 0 
        \end{array} \right ) \right \},$$
where $n = k + h$ and $X, Y, Z, T$ are $k \times k$, $k \times h$, $h \times k$ and $h \times h$ matrices, respectively.

In case $h = 2l$ it is possible to define a superinvolution $osp$, called orthosymplectic superinvolution, as follows:
 $$\left (\begin{array}{cc}
     X & Y \\
     Z & T
  \end{array} \right )^{osp} 
  = 
  \left (\begin{array}{cc}
    I_k  & 0 \\
    0 & Q
  \end{array} \right )^{-1} 
  \left (\begin{array}{cc}
    X  & -Y \\
    Z  & T
  \end{array} \right)^t
  \left (\begin{array}{cc}
    I_k  & 0 \\
    0 & Q
  \end{array} \right )
  =
  \left (\begin{array}{cc}
    X^t  & Z^tQ \\
    QY^t & -QT^tQ
  \end{array} \right),$$
where $I_k$ is the $k \times k$ identity matrix, 
$Q = \left (\begin{array}{cc}
    0 & I_l \\
   -I_l  & 0
\end{array} \right )$ and $t$ is the usual transpose.
\\We consider the particular case $k = 1$ and $l = 1$:
    $$M_{1, 2}(F) = \Set {\left (\begin{array}{c|cc}
    a & b & c \\
    \hline d & e & f \\
    g & h & i 
\end{array} \right ) | a, b, c, d, e, f, g, h, i \in F}
\ \ \text{and} \ \
    \left (\begin{array}{c|cc}
    a & b & c \\
    \hline d & e & f \\
    g & h & i 
\end{array} \right )^{osp}
=
\left (\begin{array}{c|cc}
    a & -g & d \\
    \hline c & i & -f \\
    -b & -h & e
\end{array} \right ).$$
In this case, if we denote by $e_{i, j}$ the usually elementary matrix,
    $$(M_{1, 2}(F))_0^+ = 
    span_F\{e_{11}, e_{22} + e_{33}\}, \ \ \
    (M_{1, 2}(F))_0^- = 
    span_F\{e_{22} - e_{33}, e_{23}, e_{32}\},$$
    $$(M_{1, 2}(F))_1^+ = 
    span_F\{e_{12} - e_{31}, e_{13} + e_{21}\}, \ \ \
    (M_{1, 2}(F))_1^- = 
    span_F\{e_{12} + e_{31}, e_{13} - e_{21}\}.$$

Since dim$(M_{1, 2}(F))_0^+ = 2$, dim$(M_{1, 2}(F))_0^- = 3$, dim$(M_{1, 2}(F))_1^+ = 2$ and dim$(M_{1, 2}(F))_1^- = 2$, then, by \eqref{eqn:cocarattere2} we obtain 

\begin{equation}\label{eqn:cocarattere}
    \chi_n^{*}(M_{1,2}(F)) = \sum_{\substack{\langle \lambda \rangle \vdash n, \\ h(\lambda(1)) \leq 2, \ h(\lambda(2)) \leq 3, \\ h(\lambda(3)) \leq 2, \ h(\lambda(4)) \leq 2}} m_{\langle \lambda \rangle} \chi_ {\langle \lambda \rangle}.
\end{equation}

\section{Classifying the $*$-identities of degree $\leq 3$}

In this section we determine all the $*$-identities of $M_{1,2}(F)$ of degree $\leq 3$.

We recall that, if we consider an algebra $A$ with involution $\phi$, $(A, \phi)$, we denote by $F \langle X, \phi \rangle$ the free associative algebra with involution generated by $X = \{x_1, x_2, \dots \}$ over $F$. So, an element $f(x_1, x_1^\phi, \dots , x_n, x_n^\phi) \in F \langle X, \phi \rangle$ is a $\phi$-polynomial identity for $A$ if $f(a_1, a_1^\phi, \dots , a_n, a_n^\phi) = 0$ for all substitutions $a_1, \dots, a_n \in A$. Then, we define also as $Id(A, \phi)$, the set of all $\phi$-polynomial identities of $A$, that is a $T^\phi$-ideal  of $F \langle X, \phi \rangle$, i.e., an ideal invariant under all endomorphisms of $F \langle X, \phi \rangle$ commuting with $\phi$.

Notice that $((M_{1, 2}(F))_0, osp)$ is in particular an algebra with involution and $((M_{1, 2}(F))_0, osp) \cong (F \oplus M_2(F), \phi_s)$, as algebras with involution, where $\phi_s$ is defined as $\phi_s(b + B) = b + B^s$ with $s$ the symplectic involution on $M_2(F)$. So we obtain that

\begin{remark}\label{rmk:Id}
$Id((M_{1, 2}(F))_0, osp) = Id(M_2(F), s)$.   
\end{remark}

Given $*$-polynomials $f_1, \dots , f_l \in F \langle Y \cup Z, * \rangle$, we denote by $\langle f_1, \dots , f_t \rangle_{T_2^*}$ the $T_2^*$-ideal of $F \langle Y \cup Z, * \rangle$ generated by $f_1, \dots , f_l$. 
We recall that a $*$-identity $g$ is a consequence of the $*$-polynomial identities $f_i$, with $i = 1, \dots , l$, if $g \in \langle f_1, \dots , f_l \rangle_{T_2^*}$.

By \cite[Theorem 1]{L}, we know that $Id((M_2(F), s))$ is generated, as a $T^\phi$-ideal, by $[y, x] := yx - xy$, where $y$ denotes a symmetric variable and $x$ denotes any variable in $X$. Hence we get.

\begin{remark}\label{rmk:Id(y)}
     Every $*$-identity on variables of degree zero follows from the $*$-identity $[y^+, y]$, where $y$ denotes a variable of homogeneous degree zero.
\end{remark}

Now, we prove the following.

\begin{remark}\label{rmk:Id(2)}
If $f$ is a $*$-identity of $M_{1,2}(F)$ of degree $2$, then $f = \alpha [y_1^+, y_2^+]$ or $f = \alpha [y^+, y^-]$, with $\alpha \in F$.
\end{remark}

\begin{proof}
Since $F$ is an infinite field, every $T_2^*$-ideal is generated by its multihomogeneous $*$-polynomials (see \cite{GZ}). Hence, we may assume that $f$ is a multihomogeneous $*$-polynomial of degree $2$. If $f$ is a $*$-identity on variables of degree zero, the result follows from Remark \ref{rmk:Id(y)}. Now, assume that at least one variable of degree one appears in $f$. 
It is easy to see, by making suitable evaluations of the variables of all possible $*$-polynomials, that $f$ must be the zero $*$-polynomial.
\end{proof}

We will establish a relation between $*$-identities of $M_{1,2}(F)$ in terms of variables in $Z^+$ and in $Z^-$.

We define the algebra isomorphism $\tilde{\varphi} : F \langle Y \cup Z, * \rangle \longrightarrow F \langle Y \cup Z, * \rangle$ given by
    $$\tilde{\varphi}(y^+) = y^+, \ \ \ \tilde{\varphi}(y^+) = y^-, \ \ \ \tilde{\varphi}(z^+) = z^-, \ \ \ \tilde{\varphi}(z^-) = z^+.$$

We let $\varphi : M_{1,2}(F) \longrightarrow M_{1,2}(F)$ be the linear map defined as the extension of 
 \begin{center}
     $e_{11} \mapsto e_{11}, \ e_{12} \mapsto e_{12}, \ e_{21} \mapsto - e_{21}, \ e_{22} \mapsto e_{22}, \ e_{13} \mapsto e_{13}, \ e_{31} \mapsto - e_{31}, \ e_{23} \mapsto e_{23}, \ e_{32} \mapsto e_{32}, \ e_{33} \mapsto e_{33}.$
 \end{center}
\begin{remark}
We observe that $\varphi$ is a linear isomorphism such that 
         $$\varphi [(M_{1,2}(F))_0^+] = (M_{1,2}(F))_0^+; \ \ \ \varphi [(M_{1,2}(F))_0^-] = (M_{1,2}(F))_0^-;$$
         $$\varphi [(M_{1,2}(F))_1^+] = (M_{1,2}(F))_1^-; \ \ \ \varphi [(M_{1,2}(F))_1^-] = (M_{1,2}(F))_1^+.$$
\end{remark}

Now, the following result holds.

\begin{theorem}\label{th:id}
Let $f \in F \langle Y \cup Z, * \rangle$. Then $f \in Id_2^*(M_{1,2}(F))$ if and only if $\tilde{\varphi}(f) \in Id_2^*(M_{1,2}(F))$.
\end{theorem}

\begin{proof}
It can be proved as an extension of \cite[Corollary 3.6]{CSV}.
\end{proof}

\begin{remark}
From now on, we write the symbols $\tilde{}$ and $\tilde{\tilde{}}$ in order to indicate alternation on a specific set of variables.

For example, if we write $\tilde{x}_1 \tilde{\tilde{x}}_1 x_4 \tilde{x}_2 \tilde{x}_3 \tilde{\tilde{x}}_2$, its corresponding $*$-polynomial is the following: 
     $$\sum_{\sigma \in S_3, \tau \in S_2} sgn(\sigma) sgn(\tau) x_{\sigma(1)} x_{\tau(1)} x_4 x_{\sigma(2)} x_{\sigma(3)} x_{\tau(2)}.$$

\end{remark}

For any multipartition $(n_1, n_2, n_3, n_4) \vdash 3$, we define $W_{(n_1, n_2, n_3, n_4)}$ as the subspace of $F_3^3$, formed by multihomogeneous $*$-polynomials of total degree $n_1$ in the variables $y_1^+, y_2^+, y_3^+$, of total degree $n_2$ in the variables $y_1^-, y_2^-, y_3^-$, of total degree $n_3$ in the variables $z_1^+, z_2^+, z_3^+$ and of total degree $n_4$ in the variables $z_1^-, z_2^-, z_3^-$. 
We can act on $Id_2^*(M_{1,2}(F)) \cap W_{(n_1, n_2, n_3, n_4)}$ with $GL_{n_1} \times GL_{n_2} \times GL_{n_3} \times GL_{n_4}$ and let $Id_2^*(M_{1,2}(F)) \cap W_{(n_1, n_2, n_3, n_4)} \cong \bigoplus_{\langle \lambda \rangle \vdash \langle 3 \rangle} m_{\langle \lambda \rangle} W^{\langle \lambda \rangle}$ be the decomposition into irreducible submodules, where $W^{\langle \lambda \rangle}$ is the irreducible submodule corresponding to $\langle \lambda \rangle$ generated by the highest weight vector $f_{\langle \lambda \rangle}$.

Next we shall determine the exact value of $m_{\langle \lambda \rangle}$ as follows. According to Proposition \ref{prop:hwv}, any highest highest weight vector $f_{\langle \lambda \rangle}$  can be written uniquely as a linear combination of highest weight vectors $f_{T_{\langle \lambda \rangle}}$ corresponding to standard multitableaux. So, we write this linear combination explicitly and we evaluate it into $3 \times 3$ generic matrices with superinvolution by imposing that it must be a $*$-polynomial identity of the algebra $M_{1, 2}(F)$. We obtain a system where the coefficients of the linear combination are the unknowns. It can be completely solved by making evaluations of the generic matrices with superinvolution in $M_{1, 2}(F)$. Then we prove that we obtain $*$-identities of the algebra. If we obtain different highest weight vectors corresponding to the same multipartition, we check that they are linearly independent. The maximal number of linearly independent highest weight vectors will be the multiplicity of the corresponding $GL_{n_1} \times GL_{n_2} \times GL_{n_3} \times GL_{n_4}$-module.

Let us explain with an example. Consider the composition $(2, 1, 0, 0) \vdash 3$.
We have the multipartitions $\langle \lambda_1 \rangle = ((2), (1), \emptyset, \emptyset)$ and $\langle \lambda_2 \rangle = ((1, 1), (1), \emptyset, \emptyset)$.

For $\langle \lambda_1 \rangle = ((2), (1), \emptyset, \emptyset)$, we have the following standard multitableaux:
$$\left ( \ \begin{array}{|c|c|}
    \hline 1 & 2 \\
    \hline
\end{array} \ , \ \begin{array}{|c|}
 \hline   3  \\
     \hline
\end{array} \ , \ \emptyset \ , \ \emptyset \right ),
\left ( \ \begin{array}{|c|c|}
    \hline 1 & 3 \\
    \hline
\end{array} \ , \ \begin{array}{|c|}
 \hline   2  \\
     \hline
\end{array} \ , \ \emptyset \ , \ \emptyset \right )
\text{ and }
\left ( \ \begin{array}{|c|c|}
    \hline 2 & 3 \\
    \hline
\end{array} \ , \ \begin{array}{|c|}
 \hline   1  \\
     \hline
\end{array} \ , \ \emptyset \ , \ \emptyset \right ),$$

with highest weight vectors given by:
   $$f_{\bar{T}_{\langle \lambda_1 \rangle}} = (y_1^+)^2 y_1^-, \ \ \ f_{\bar{T}_{\langle \lambda_1 \rangle}} (23)^{-1} = y_1^+ y_1^- y_1^+ \ \ \ \text{and} \ \ \ f_{\bar{T}_{\langle \lambda_1 \rangle}} (123)^{-1} = y_1^- (y_1^+)^2.$$

For $\langle \lambda_2 \rangle = ((1, 1), (1), \emptyset, \emptyset)$, we have the following standard multitableaux:

$$\left ( \ \begin{array}{|c|}
    \hline 1 \\
    \hline 2 \\
    \hline
\end{array} \ , \ \begin{array}{|c|}
 \hline   3  \\
     \hline
\end{array} \ , \ \emptyset \ , \ \emptyset \right ),
\left ( \ \begin{array}{|c|}
    \hline 1 \\
    \hline 3 \\
    \hline
\end{array} \ , \ \begin{array}{|c|}
 \hline   2  \\
     \hline
\end{array} \ , \ \emptyset \ , \ \emptyset \right )
\text{ and }
\left ( \ \begin{array}{|c|}
    \hline 2 \\
    \hline 3 \\
    \hline
\end{array} \ , \ \begin{array}{|c|}
 \hline   1  \\
     \hline
\end{array} \ , \ \emptyset \ , \ \emptyset \right ),$$

with highest weight vectors given by:
    $$g_{\bar{T}_{\langle \lambda_2 \rangle}} = \tilde{y}_1^+ \tilde{y}_2^+ y_1^-, \ \ \ g_{\bar{T}_{\langle \lambda_2 \rangle}} (23)^{-1} = \tilde{y}_1^+ y_1^- \tilde{y}_2^+ \ \ \ \text{and} \ \ \ g_{\bar{T}_{\langle \lambda_2 \rangle}} (123)^{-1} = y_1^- \tilde{y}_1^+ \tilde{y}_2^+.$$

\noindent We consider the decomposition $Id_2^*(M_{1,2}(F)) \cap W_{(2, 1, 0, 0)} \cong \beta_1 W^{((2), (1), \emptyset, \emptyset)} \oplus \beta_2 W^{((1, 1), (1), \emptyset, \emptyset)}$ and now we compute the multiplicities $\beta_1$ and $\beta_2$.

The highest weight vector which generates $W^{((2), (1), \emptyset, \emptyset)}$ is $f_{\langle \lambda_1 \rangle} = \alpha_1 (y_1^+)^2 y_1^- + \alpha_2 y_1^+ y_1^- y_1^+ + \alpha_3 y_1^- (y_1^+)^2$ that belongs to $Id_2^*(M_{1,2}(F)),$ for some $\alpha_1, \alpha_2, \alpha_3 \in F$. Considering the substitutions $y_1^+ = e_{22} + e_{33}$ and $y_1^- = e_{23}$, we get $(\alpha_1 + \alpha_2 + \alpha_3) e_{23} = 0,$ which implies that $\alpha_3 = -\alpha_1 - \alpha_2.$
It follows that $f_{\langle \lambda_1 \rangle} = \alpha_1 [(y_1^+)^2, y_1^-] + \alpha_2 [y_1^+, y_1^- y_1^+]$ and, since $[(y_1^+)^2, y_1^-]$ and $[y_1^+, y_1^- y_1^+]$ are $*$-identities and they are linearly independent, we get $\beta_1 = 2.$

The highest weight vector which generates $W^{((1, 1), (1), \emptyset, \emptyset)}$ is $f_{\langle \lambda_2 \rangle} = \alpha_4 \tilde{y}_1^+ \tilde{y}_2^+ y_1^- + \alpha_5 \tilde{y}_1^+ y_1^- \tilde{y}_2^+ + \alpha_6 y_1^- \tilde{y}_1^+ \tilde{y}_2^+ \in Id_2^*(M_{1,2}(F)),$ for some $\alpha_4, \alpha_5, \alpha_6 \in F$. We observe that $\tilde{y}_1^+ \tilde{y}_2^+ y_1^- = [y_1^+, y_2^+] y_1^- \equiv 0$ and $y_1^- \tilde{y}_1^+ \tilde{y}_2^+ = y_1^- [y_1^+, y_2^+] \equiv 0,$ because $[y_1^+, y_2^+] \equiv 0$. 
It is clear that also $\tilde{y}_1^+ y_1^- \tilde{y}_2^+ \equiv 0.$ These three identities are also linearly independent, so $\beta_2 = 3.$
 
\noindent In conclusion:
    $$Id_2^* (M_{1, 2}(F)) \cap W_{(2, 1, 0, 0)} \cong 2W^{((2), (1), \emptyset, \emptyset)} \oplus 3W^{((1, 1), (1), \emptyset, \emptyset)}$$
with    
    $$f'_{\langle \lambda_1 \rangle} = [(y_1^+)^2, y_1^-], \ \ \ f''_{\langle \lambda_1 \rangle} = [y_1^+, y_1^- y_1^+] \ \ \ \text{and} \ \ \ f'_{\langle \lambda_2 \rangle} = [y_1^+, y_2^+] y_1^-, \ \ \ f''_{\langle \lambda_2 \rangle} = \tilde{y}_1^+ y_1^- \tilde{y}_2^+, \ \ \ f'''_{\langle \lambda_2 \rangle} = y_1^- [y_1^+, y_2^+].$$

Proceeding in the same way for the remaining compositions of $3$ and their corresponding multipartitions, we get the following.

\begin{theorem}\label{th:decompositions}
The following decompositions are valid:

 \begin{enumerate}
     
     \item $Id_2^* (M_{1, 2}(F)) \cap W_{(3, 0, 0, 0)} \cong 2W^{((2, 1), \emptyset, \emptyset, \emptyset)} \oplus W^{((1^3), \emptyset, \emptyset, \emptyset)}$ with highest weight vectors
     \medskip
          \begin{center}
              $f'_{((2, 1), \emptyset, \emptyset, \emptyset)} = [y_1^+, y_2^+] y_1^+$, \ \ \ $f''_{((2, 1), \emptyset, \emptyset, \emptyset)} = \tilde{y}_1^+ y_1^+ \tilde{y}_2^+ = [(y_1^+)^2, y_2^+],$
              \\$f_{((1^3), \emptyset, \emptyset, \emptyset)} = \tilde{y}_1^+ \tilde{y}_2^+ \tilde{y}_3^+ = St_3(y_1^+, y_2^+, y_3^+)$;
          \end{center}
        \medskip
      
     \item $Id_2^* (M_{1, 2}(F)) \cap W_{(0, 3, 0, 0)} \cong W^{(\emptyset, (2, 1), \emptyset, \emptyset)}$ with highest weight vectors
     \medskip
      \begin{center}
          $f_{(\emptyset, (2, 1), \emptyset, \emptyset)} = \tilde{y}_1^- y_1^- \tilde{y}_2^- = [(y_1^-)^2, y_2^-]$;
      \end{center}
      \medskip

     \item $Id_2^* (M_{1, 2}(F)) \cap W_{(2, 1, 0, 0)} \cong 2W^{((2), (1), \emptyset, \emptyset)} \oplus 3W^{((1^2), (1), \emptyset, \emptyset)}$ with highest weight vectors
     \medskip
      \begin{center}
      $f'_{((2), (1), \emptyset, \emptyset)} = [(y^+)^2, y^-]$, \ \ \ $f''_{((2), (1), \emptyset, \emptyset)} = [y^+, y^- y^+]$, 
      \\$f'_{((1^2), (1), \emptyset, \emptyset)} = [y_1^+, y_2^+] y^-$, \ \ \ $f''_{((1^2), (1), \emptyset, \emptyset)} = \tilde{y}_1^+ y^- \tilde{y}_2^+$, \ \ \ $f'''_{((1^2), (1), \emptyset, \emptyset)} = y^- [y_1^+, y_2^+]$;
      \end{center}
      \medskip

     \item $Id_2^* (M_{1, 2}(F)) \cap W_{(1, 2, 0, 0)} \cong 2W^{((1), (2), \emptyset, \emptyset)} \oplus 2W^{((1), (1^2), \emptyset, \emptyset)}$ with highest weight vectors
     \medskip
      \begin{center}
          $f'_{((1), (2), \emptyset, \emptyset)} = [y^+, (y^-)^2]$, \ \ \ $f''_{((1), (2), \emptyset, \emptyset)} = y^- [y^+, y^-]$,
          \\$f'_{((1), (1^2), \emptyset, \emptyset)} = [y^+, \tilde{y}_1^- \tilde{y}_2^-]$, \ \ \ $f''_{((1), (1^2), \emptyset, \emptyset)} = \tilde{y}_1^- [y^+, \tilde{y}_2^-]$;
      \end{center}
      \medskip

     \item $Id_2^* (M_{1, 2}(F)) \cap W_{(2, 0, 1, 0)} \cong 2W^{((1^2), \emptyset, (1), \emptyset)}$ with highest weight vectors
     \medskip
      \begin{center}
          $f'_{((1^2), \emptyset, (1), \emptyset)} = [y_1^+, y_2^+] z^+$, \ \ \ $f''_{((1^2), \emptyset, (1), \emptyset)} = z^+ [y_1^+, y_2^+]$;
      \end{center}
      \medskip
 
     \item $Id_2^* (M_{1, 2}(F)) \cap W_{(2, 0, 0, 1)} \cong 2W^{((1^2), \emptyset, \emptyset, (1))}$ with highest weight vectors
     \medskip
      \begin{center}
          $f'_{((1^2), \emptyset, \emptyset, (1))} = [y_1^+, y_2^+] z^-$, \ \ \ $f''_{((1^2), \emptyset, \emptyset, (1))} = z^- [y_1^+, y_2^+]$;
      \end{center}
      \medskip

     \item $Id_2^* (M_{1, 2}(F)) \cap W_{(1, 1, 1, 0)} \cong 2W^{((1), (1), (1), \emptyset)}$ with highest weight vectors
     \medskip
      \begin{center}
          $f'_{((1), (1), (1), \emptyset)} = [y^+, y^-] z^+$, 
          \ \ \ $f''_{((1), (1), (1), \emptyset)} = z^+ [y^+, y^-]$;
      \end{center}
      \medskip
      
     \item $Id_2^* (M_{1, 2}(F)) \cap W_{(1, 1, 0, 1)} \cong 2W^{((1), (1), \emptyset, (1))}$ with highest weight vectors
     \medskip
      \begin{center}
          $f'_{((1), (1), \emptyset, (1))} = [y^+, y^-] z^-$, 
          \ \ \ $f''_{((1), (1), \emptyset, (1))} = z^- [y^+, y^-]$;
      \end{center}
      \medskip

     \item $Id_2^* (M_{1, 2}(F)) \cap W_{(1, 0, 2, 0)} \cong W^{((1), \emptyset, (2), \emptyset)} \oplus W^{((1), \emptyset, (1^2), \emptyset)}$ with highest weight vectors
     \medskip
      \begin{center}
          $f_{((1), \emptyset, (2), \emptyset)} = [y^+, (z^+)^2]$,
          \\$f_{((1), \emptyset, (1^2), \emptyset)} = [y^+, \tilde{z}_1^+ \tilde{z}_2^+]$;
      \end{center}
      \medskip
      
     \item $Id_2^* (M_{1, 2}(F)) \cap W_{(1, 0, 0, 2)} \cong W^{((1), \emptyset, \emptyset, (2))} \oplus W^{((1), \emptyset, \emptyset, (1^2))}$ with highest weight vectors
     \medskip
      \begin{center}
          $f_{((1), \emptyset, \emptyset, (2))} = [y^+, (z^-)^2]$,
          \\$f_{((1), \emptyset, \emptyset, (1^2))} = [y^+, \tilde{z}_1^- \tilde{z}_2^-]$;
      \end{center}
      \medskip

     \item $Id_2^* (M_{1, 2}(F)) \cap W_{(1, 0, 1, 1)} \cong 2W^{((1), \emptyset, (1), (1))}$ with highest weight vectors
     \medskip
      \begin{center}
          $f'_{((1), \emptyset, (1), (1))} = [y^+, z^+ z^-]$, 
          \ \ \ $f''_{((1), \emptyset, (1), (1))} = [y^+, z^- z^+]$;
      \end{center}
      \medskip
      
     \item $Id_2^* (M_{1, 2}(F)) \cap W_{(0, 0, 3, 0)} \cong W^{(\emptyset, \emptyset, (3), \emptyset)} \oplus W^{(\emptyset, \emptyset, (1^3), \emptyset)}$ with highest weight vectors
     \medskip
      \begin{center}
          $f_{(\emptyset, \emptyset, (3), \emptyset)} = (z^+)^3$,
          \\$f_{(\emptyset, \emptyset, (1^3), \emptyset)} = \tilde{z}_1^+ \tilde{z}_2^+ \tilde{z}_3^+ = St_3(z_1^+, z_2^+, z_3^+)$;
      \end{center}
      \medskip
      
     \item $Id_2^* (M_{1, 2}(F)) \cap W_{(0, 0, 0, 3)} \cong W^{(\emptyset, \emptyset, \emptyset, (3))} \oplus W^{(\emptyset, \emptyset, \emptyset, (1^3))}$ with highest weight vectors
     \medskip
      \begin{center}
          $f_{(\emptyset, \emptyset, \emptyset, (3))} = (z^-)^3$,
          \\$f_{(\emptyset, \emptyset, \emptyset, (1^3))}  = \tilde{z}_1^- \tilde{z}_2^- \tilde{z}_3^- = St_3(z_1^-, z_2^-, z_3^-)$;
      \end{center}
      \medskip
      
     \item $Id_2^* (M_{1, 2}(F)) \cap W_{(0, 2, 1, 0)} \cong W^{(\emptyset, (2), (1), \emptyset)} \oplus W^{(\emptyset, (1^2), (1), \emptyset)}$ with highest weight vectors
     \medskip
      \begin{center}
          $f_{(\emptyset, (2), (1), \emptyset)} = y^- z^+ y^-$,
          \\$f_{(\emptyset, (1^2), (1), \emptyset)} = \tilde{y}_1^- z^+ \tilde{y}_2^-$;
      \end{center}
      \medskip
      
     \item $Id_2^* (M_{1, 2}(F)) \cap W_{(0, 2, 0, 1)} \cong W^{(\emptyset, (2), \emptyset, (1))} \oplus W^{(\emptyset, (1^2), \emptyset, (1))}$ with highest weight vectors
     \medskip
      \begin{center}
          $f_{(\emptyset, (2), \emptyset, (1))} = y^- z^- y^-$,
          \\$f_{(\emptyset, (1^2), \emptyset, (1))} = \tilde{y}_1^- z^- \tilde{y}_2^-$;
      \end{center}
      \medskip
      
     \item $Id_2^* (M_{1, 2}(F)) \cap W_{(0, 1, 2, 0)} \cong 2W^{(\emptyset, (1), (1^2), \emptyset)}$ with highest weight vectors
     \medskip
      \begin{center}
          $f'_{(\emptyset, (1), (1^2), \emptyset)} = \tilde{z}_1^+ y^- \tilde{z}_2^+$, \ \ \ $f''_{(\emptyset, (1), (1^2), \emptyset)} = [y^-, \tilde{z}_1^+ \tilde{z}_2^+]$;
      \end{center}
      \medskip
      
     \item $Id_2^* (M_{1, 2}(F)) \cap W_{(0, 1, 0, 2)} \cong 2W^{(\emptyset, (1), \emptyset, (1^2))}$ with highest weight vectors
     \medskip
      \begin{center}
          $f'_{(\emptyset, (1), \emptyset, (1^2))} = \tilde{z}_1^- y^- \tilde{z}_2^-$, \ \ \ $f''_{(\emptyset, (1), \emptyset, (1^2))} = [y^-, \tilde{z}_1^- \tilde{z}_2^-]$;
      \end{center}
      \medskip
      
     \item $Id_2^* (M_{1, 2}(F)) \cap W_{(0, 0, 2, 1)} \cong W^{(\emptyset, \emptyset, (2), (1))} \oplus W^{(\emptyset, \emptyset, (1^2), (1))}$ with highest weight vectors
     \medskip
      \begin{center}
          $f_{(\emptyset, \emptyset, (2), (1))} = (z^+)^2 z^- - z^+ z^- z^+ + z^- (z^+)^2$, 
          \\$f_{(\emptyset, \emptyset, (1^2), (1))} = \tilde{z}_1^+ \tilde{z}_2^+ z^- + \tilde{z}_1^+ z^- \tilde{z}_2^+ + z^- \tilde{z}_1^+ \tilde{z}_2^+$;
      \end{center}
      \medskip
      
     \item $Id_2^* (M_{1, 2}(F)) \cap W_{(0, 0, 1, 2)} \cong W^{(\emptyset, \emptyset, (1), (2))} \oplus W^{(\emptyset, \emptyset, (1), (1^2))}$ with highest weight vectors
     \medskip
      \begin{center}
          $f_{(\emptyset, \emptyset, (1), (2))} = (z^-)^2 z^+ - z^- z^+ z^- + z^+ (z^-)^2$, 
          \\$f_{(\emptyset, \emptyset, (1), (1^2))} = \tilde{z}_1^- \tilde{z}_2^- z^+ + \tilde{z}_1^- z^+ \tilde{z}_2^- + z^+ \tilde{z}_1^- \tilde{z}_2^-$;
      \end{center}
      \medskip
      
     \item $Id_2^* (M_{1, 2}(F)) \cap 2W_{(0, 1, 1, 1)} \cong W^{(\emptyset, (1), (1), (1))}$ with highest weight vectors
     \medskip
      \begin{center}
          $f'_{(\emptyset, (1), (1), (1))} = z^+ y^- z^- + z^- y^- z^+$, \ \ \ $f''_{(\emptyset, (1), (1), (1))} = [y^-, z^+ \circ z^-]$.
      \end{center}

 \end{enumerate}
 
\end{theorem}

Now, we consider the following set:

\begin{center}
    $\mathcal{I} = \{[y^+, y], \ \ \ y_1^- z^+ y_2^-, \ \ \ y_1^- z^- y_2^-, \ \ \ z_1^+ z_2^+ z_3^+ + z_2^+ z_3^+ z_1^+ + z_3^+ z_1^+ z_2^+, \ \ \ z_1^- z_2^- z_3^- + z_2^- z_3^- z_1^- + z_3^- z_1^- z_2^-, \ \ \ \tilde{z}_1^+ y^- \tilde{z}_2^+, \ \ \ \tilde{z}_1^- y^- \tilde{z}_2^-, \ \ \ z_1^+ z_2^+ z^- - z_2^+ z^- z_1^+ + z^- z_1^+ z_2^+, \ \ \ z_1^- z_2^- z^+ - z_2^- z^+ z_1^- + z^+ z_1^- z_2^-, \ \ \ z^+y^-z^- + z^-y^-z^+\}$.
\end{center}

It is easy to prove that each $*$-polynomial of this set is a $*$-identity of $M_{1, 2}(F)$, then $\mathcal{I} \subseteq Id_2^*(M_{1, 2}(F))$.

\begin{theorem}\label{id}
Let $f \in Id_2^*(M_{1, 2}(F))$ of degree $\leq 3$. Then $f$ is a consequence of $*$-polynomials in the set $\mathcal{I}$.
\end{theorem}

\begin{proof}
It is obvious that $M_{1, 2}(F)$ does not satisfy any $*$-identity of degree $1$. If $f$ has degree $2$, by Remark \ref{rmk:Id(2)}, any $*$-identity follows from $[y_1^+, y_2^+]$ and $[y^+, y^-]$ which are consequences of $[y^+, y]$ and we are done in this case. Now assume that $f$ has degree $3$ and that $f$ is multihomogeneous, as we may. So, $f \in W_{(n_1, n_2, n_3, n_4)}$, for some composition $(n_1, n_2, n_3, n_4)$ of $3$. In order to complete the proof we shall show that any highest weight vector corresponding to the multipartition $(n_1, n_2, n_3, n_4)$ given in Theorem \ref{th:decompositions} is a consequence of the $*$-polynomials in $\mathcal{I}$. 
We enumerate such highest weight vectors in this way:

   \begin{multicols}{2}
    \begin{enumerate}
        \item[(a)] $[y^+, y]$,
        \item[(b)] $z_1^+ z_2^+ z_3^+ + z_2^+ z_3^+ z_1^+ + z_3^+ z_1^+ z_2^+$,
        \item[(c)] $z_1^- z_2^- z_3^- + z_2^- z_3^- z_1^- + z_3^- z_1^- z_2^-$,
        \item[(d)]  $y_1^- z^+ y_2^-$,
        \item[(e)] $y_1^- z^- y_2^-$,
        \item[(f)] $\tilde{z}_1^+ y^- \tilde{z}_2^+$,
        \item[(g)] $\tilde{z}_1^- y^- \tilde{z}_2^-$,
        \item[(h)] $z_1^+ z_2^+ z^- - z_2^+ z^- z_1^+ + z^- z_1^+ z_2^+$,
        \item[(i)] $z_1^- z_2^- z^+ - z_2^- z^+ z_1^- + z^+ z_1^- z_2^-$,
        \item[(j)] $z^+y^-z^- + z^-y^-z^+$.
    \end{enumerate}
    \end{multicols}

Clearly, if $f$ is a $*$-identity in variables of degree zero, by Remark \ref{rmk:Id(y)}, it is a consequence of $[y^+, y]$. Hence, the highest weight vectors in (1), (2), (3) and (4) follow from (a). Also the highest weight vectors in (5), (6), (7) and (8) follow from (a). Moreover, since the product of two variables of degree one is a variable of degree zero, also the highest weight vectors in (9), (10) and (11) follow from (a). It is immediate to see that the highest weight vectors in (12), (13), (14) and (15) follow from (b), (c), (d) and (e), respectively. In (16) and (17) we find the $*$-polynomials (f) and (g) and consequences of (a). The highest weight vectors in (18) and (19) follow from (h) and (i), respectively. Finally, the $*$-polynomials in (20) follow from (j) and (a).
\end{proof}

\section{On the $*$-cocharacter of $M_{1,2}(F)$}

In this section we determine the decomposition of the $*$-cocharacter of $M_{1, 2}(F)$ with multiplicity $m_{\langle \lambda \rangle} \neq 0$.

We recall that

\begin{equation}\label{eqn:ultimococarattere}
   \chi_n^{*}(M_{1,2}(F)) = \sum_{\substack{\langle \lambda \rangle \vdash n, \\ h(\lambda(1)) \leq 2, \ h(\lambda(2)) \leq 3, \\ h(\lambda(3)) \leq 2, \ h(\lambda(4)) \leq 2}} m_{\langle \lambda \rangle} \chi_ {\langle \lambda \rangle},
\end{equation}

where $m_{\langle \lambda \rangle} \geq 0$ is the multiplicity corresponding to the irreducible $\mathbb{H}_n$-character $\chi_{\langle \lambda \rangle}$, with $\langle \lambda \rangle = (\lambda(1), \dots, \lambda(4)) \vdash n$.

In \cite[Theorem 4.1b]{DG}, Drensky and Giambruno determined the decomposition of the cocharacter of the algebra $M_2(F)$ endowed with the simplectic involution. They proved the following result.

\begin{theorem}\label{th:s}
The $\mathbb{Z}_2 \sim S_n$-cocharacter  of $M_2(F)$ endowed with the symplectic involution $s$ is 
  \begin{equation}
     \chi_n(M_2(F), s) = \sum_{\substack{n = n_1 + n_2, \\ \lambda \vdash n_1, \ \mu \vdash n_2}} a_{\lambda, \mu} \chi_{\lambda, \mu},
  \end{equation}
where $\lambda = (\lambda_1)$, $\mu = (\mu_1, \mu_2, \mu_3)$ and $a_{\lambda, \mu} = 1$.
\end{theorem}

An immediate consequence of Theorem \ref{th:s} and Remark \ref{rmk:Id} is the following.

\begin{proposition}\label{prop:12}
If $\langle \lambda \rangle = (\lambda(1), \lambda(2), \emptyset, \emptyset)$ in (\ref{eqn:ultimococarattere}), then $m_{\langle \lambda \rangle} \neq 0$ if and only if $h(\lambda(1)) \leq 1$.
\end{proposition}

Now we consider only the multiplicities $m_{\langle \lambda \rangle}$ in (\ref{eqn:ultimococarattere}) with $\lambda(2) = \lambda(4) = \emptyset$.

\begin{proposition}\label{prop:13}
If $\langle \lambda \rangle = (\lambda(1), \emptyset, \lambda(3), \emptyset)$ in (\ref{eqn:ultimococarattere}), with $\lambda(3) = (w_1 + w_2, w_2) \neq \emptyset$, then $m_{\langle \lambda \rangle} \neq 0$ if and only if  $w_1 \leq 2$.
\end{proposition}

\begin{proof}
In order to prove that $m_{\langle \lambda \rangle} \neq 0$, by Theorem \ref{th:cocharacter} and Theorem \ref{th:hwv}, we just need to show that there exists a Young multitableau $T_{(\lambda(1), \emptyset, \lambda(3), \emptyset)}$ such that the corresponding highest weight vector $f_{T_{(\lambda(1), \emptyset, \lambda(3), \emptyset)}}$ is not a $*$-identity for the algebra. Consider the elements $R_1 = e_{11}$, $R_2 = e_{22} + e_{33}$, $N_1 = e_{12} - e_{31}$, $N_2 = e_{13} + e_{21}$.

First, we suppose $h(\lambda(1)) \leq 1$ and we consider the initial standard multitableau $\bar{T}_{(\lambda(1), \emptyset, \lambda(3), \emptyset)}$ and the corresponding highest weight vector
    $$f_{\bar{T}_{(\lambda(1), \emptyset, \lambda(3), \emptyset)}}(y_1^+, z_1^+, z_2^+) = (y_1^+)^{\alpha_1}[z_1^+, z_2^+]^{w_2}(z_1^+)^{w_1}.$$
Then
    $$f_{\bar{T}_{(\lambda(1), \emptyset, \lambda(3), \emptyset)}}(R_2, N_1, N_2) = (e_{22} +e_{33})^{\alpha_1}(2^{w_2} e_{11} + (-1)^{w_2} e_{22} + (-1)^{w_2}e_{33})(e_{12} - e_{31})^{w_1} \neq 0.$$

Now assume that $h(\lambda(1)) = 2$ and distinguish three different cases.
 \begin{itemize}
 
    \item $w_1 = 0$. 
    \\We consider the following standard multitableaux 
     \begin{center}
        $T_{(\lambda(1), \emptyset, \lambda(3), \emptyset)} =$
     \end{center}
     \begin{center}
        $\left ( \ 
        \begin{array}{ccccc}
         \hline \multicolumn{1}{|c|}{1} & \multicolumn{1}{|c|}{\ \dots \ } & \multicolumn{1}{|c|}{\alpha_2} & \multicolumn{1}{|c|}{\ \dots \ } & \multicolumn{1}{|c|}{\alpha_2 + \alpha_1} \\
         \hline \multicolumn{1}{|c|}{\substack{\alpha_2 + \alpha_1 + 2}} & \multicolumn{1}{|c|}{\ \dots \ } & \multicolumn{1}{|c|}{\substack{2\alpha_2 + \alpha_1 + 1}} \\
         \cline{1-3} 
        \end{array}
        \ , \ \emptyset \ , \ 
        \begin{array}{|c|c|c|c|}
         \hline \substack{\alpha_2 + \alpha_1 + 1} & \substack{2\alpha_2 + \alpha_1 + 3} & \ \dots \ & n - 1 \\
         \hline \substack{2\alpha_2 + \alpha_1 + 2} & \substack{2\alpha_2 + \alpha_1 + 4} & \ \dots \ & n \\
         \hline
        \end{array}
        \ , \ \emptyset \right )$
     \end{center} 
     
    \   
    
    with
        $$f_{T_{(\lambda(1), \emptyset, \lambda(3), \emptyset)}}(y_1^+, y_2^+, z_1^+, z_2^+) = \underbrace{\tilde{y}_1^+ \dots \tilde{\tilde{y}}_1^+}_{\alpha_2}(y_1^+)^{\alpha_1}\tilde{z}_1^+\underbrace{\tilde{y}_2^+ \dots \tilde{\tilde{y}}_2^+}_{\alpha_2}\tilde{z}_2^+[z_1^+, z_2^+]^{w_2 - 1}.$$
    Then
        $$f_{T_{(\lambda(1), \emptyset, \lambda(3), \emptyset)}}(R_2, R_1, N_1, N_2) = \pm \beta e_{11} \pm e_{22} \pm e_{33} \neq 0,$$
    where $\beta \in \{0, 2^{w_2} \}$. 

   \item $w_1 = 1$.
   \\If $\lambda(3) \vdash r$, we denote by $\bar{T}_{\lambda(3)}$ the initial standard tableau on the integers $\alpha_2 + \alpha_1 + 1, \dots , \alpha_2 + \alpha_1 + r$. Consider the following multitableau
    \begin{center}
       $T_{(\lambda(1), \emptyset, \lambda(3), \emptyset)} = \left ( \ 
       \begin{array}{ccccc}
        \hline \multicolumn{1}{|c|}{1} & \multicolumn{1}{|c|}{\ \dots \ } & \multicolumn{1}{|c|}{\alpha_2} & \multicolumn{1}{|c|}{\ \dots \ } & \multicolumn{1}{|c|}{\alpha_2 + \alpha_1}\\
        \hline \multicolumn{1}{|c|}{\substack{\alpha_2 + \alpha_1 + r + 1}} & \multicolumn{1}{|c|}{\ \dots \ } & \multicolumn{1}{|c|}{n}\\
        \cline{1-3} 
       \end{array}
       \ , \ \emptyset \ , \ 
       \bar{T}_{\lambda(3)}
       \ , \ \emptyset \right )$
    \end{center} 

   \ 
          
   with corresponding highest weight vector 
       $$f_{T_{(\lambda(1), \emptyset, \lambda(3), \emptyset)}}(y_1^+, y_2^+, z_1^+, z_2^+) = \underbrace{\tilde{y}_1^+ \dots \tilde{\tilde{y}}_1^+}_{\alpha_2}(y_1^+)^{\alpha_1}f_{\bar{T}_{\lambda(3)}}(z_1^+, z_2^+)\underbrace{\tilde{y}_2^+ \dots \tilde{\tilde{y}}_2^+}_{\alpha_2} =$$
       $$\underbrace{\tilde{y}_1^+ \dots \tilde{\tilde{y}}_1^+}_{\alpha_2}(y_1^+)^{\alpha_1}[z_1^+, z_2^+]^{w_2}z_1^+\underbrace{\tilde{y}_2^+ \dots \tilde{\tilde{y}}_2^+}_{\alpha_2}.$$
   Then
       $$f_{T_{(\lambda(1), \emptyset, \lambda(3), \emptyset)}}(R_2, R_1, N_1, N_2) = \pm e_{31} \pm \beta e_{12} \neq 0,$$
   where $\beta \in \{ 0, 2^{w_2} \}$.

   \item $w_1 = 2$.
   \\If $\lambda(3) \vdash r$, we consider the following standard multitableaux 

   \
    
    \begin{center}
       $T_{(\lambda(1), \emptyset, \lambda(3), \emptyset)} =$
    \end{center}
    \begin{center}
       $\left ( \ 
        \begin{array}{ccccc}
         \hline \multicolumn{1}{|c|}{1} & \multicolumn{1}{|c|}{\ \dots \ } & \multicolumn{1}{|c|}{\alpha_2} & \multicolumn{1}{|c|}{\ \dots \ } & \multicolumn{1}{|c|}{\alpha_2 + \alpha_1}\\
         \hline \multicolumn{1}{|c|}{\substack{\alpha_2 + \alpha_1 + r}} & \multicolumn{1}{|c|}{\ \dots \ } & \multicolumn{1}{|c|}{n - 1}\\
         \cline{1-3} 
        \end{array}        
        \ , \ \emptyset \ , \ 
        \begin{array}{|ccccc|}
         \hline \multicolumn{1}{|c|}{\substack{\alpha_2 + \alpha_1 + 1}} & \multicolumn{1}{|c|}{\ \dots \ } & \multicolumn{1}{|c|}{\substack{\alpha_2 + \alpha_1 + r - 3}} & \multicolumn{1}{|c|}{\substack{\alpha_2 + \alpha_1 + r - 1}} & \multicolumn{1}{|c|}{n} \\
         \hline \multicolumn{1}{|c|}{\substack{\alpha_2 + \alpha_1 + 2}} & \multicolumn{1}{|c|}{\ \dots \ } & \multicolumn{1}{|c|}{\substack{\alpha_2 + \alpha_1 + r - 2}}\\
         \cline{1-3} 
        \end{array}
        \ , \ \emptyset \right )$
    \end{center} 

   \
          
   with 
       $$f_{T_{(\lambda(1), \emptyset, \lambda(3), \emptyset)}}(y_1^+, y_2^+, z_1^+, z_2^+) = 
       \underbrace{\tilde{y}_1^+ \dots \tilde{\tilde{y}}_1^+}_{\alpha_2}(y_1^+)^{\alpha_1}[z_1^+, z_2^+]^{w_2}z_1^+\underbrace{\tilde{y}_2^+ \dots \tilde{\tilde{y}}_2^+}_{\alpha_2} z_1^+.$$
   Hence
       $$f_{T_{(\lambda(1), \emptyset, \lambda(3), \emptyset)}}(R_2, R_1, N_1, N_2) = \pm e_{32} \neq 0.$$

 \end{itemize}

Conversely, we shall prove that if $w_1 \geq 3$ then $m_{\langle \lambda \rangle} = 0$. By Theorem \ref{th:cocharacter} and Theorem \ref{th:hwv}, we have to show that for every Young multitableau $T_{\langle \lambda \rangle}$, the corresponding highest weight vector $f_{T_{\langle \lambda \rangle}}$ is a $*$-identity for the algebra. We distinguish two cases.

\textbf{Case 1.} $\lambda(1) = \emptyset$.

By the hypothesis, any monomial of $f_{T_{(\emptyset, \emptyset, \lambda(3), \emptyset)}}$ has $w_2 + w_1$ variables $z_1^+$ and $w_2$ variables $z_2^+$, so $w_1$ is the difference between the number of variables $z_1^+$ and the number of variables $z_2^+$. For any monomial $m$ in $z_1^+$ and $z_2^+$, we define $l(m)$ as the number of $z_1^+$ in $m$, $k(m)$ as the number of variables $z_2^+$ in $m$ and $d(m) = l(m) - k(m)$. It is enough to prove that any monomial $m$ with $d(m) \geq 3$ is a $*$-identity for the algebra. We shall prove it by induction on $k(m)$. If $k(m) = 0$, then $m$ is of the type $(z_1^+)^{l(m)}$, $l(m) = d(m) \geq 3$, and it is a $*$-identity, because it is a consequence of the $*$-identity (b) of Theorem \ref{id}. If $k(m) = 1$, the possible monomials $m$ are of the type $(z_1^+)^i z_2^+ (z_1^+)^{l(m) - i}$, with $i = 0, \dots, l(m)$. We observe that if $i \geq 3$ or $l(m) - i \geq 3$, then $m$ is a $*$-identity, because it contains $(z_1^+)^3$ that is a $*$-identity as we have seen before. If $i \leq 2$ and $l(m) - i \leq 2$, since $l(m) = d(m) + 1 \geq 4$, we have to consider only the case $i = 2$ and $l(m) = 4$. Hence $m = (z_1^+)^2 z_2^+ (z_1^+)^2$ and it is a $*$-identity because it is a consequence of the $*$-identity (d) of Theorem \ref{id}.

If $k(m) > 1$, we assume that any monomial $p$ with $k(p) < k(m)$ and $d(p) \geq 3$ is a $*$-identity and we shall prove that also $m$ is a $*$- identity. Let $m = z_2^+ m'$ \ or \ $m = m' z_2^+$, where $m'$ is a monomial with $k(m') = k(m) - 1$, $d(m') = l(m) - k(m) + 1 = d(m) + 1 \geq 4$. Since $k(m') < k(m)$, by the induction hypothesis $m' \equiv 0$ and so $m \equiv 0$ is a $*$-identity. Suppose $m = z_1^+ m' z_1^+$. Then if $m = z_1^+ z_2^+ \underbrace{m'' z_1^+}_r$ \ or \ $m = \underbrace{z_1^+ m''}_r z_2^+ z_1^+$, since $r$ is a monomial with $k(r) = k(m) - 1$ and $d(r) = d(m) \geq 3$, by induction hypothesis, $r \equiv 0$ and so $m \equiv 0$ is a $*$-identity. Otherwise, if $m = z_1^+ z_1^+ m'' z_1^+ z_1^+$, with $m'' = z_1^+ m'''$ or $m'' = m''' z_1^+$, since $(z_1^+)^3 \equiv 0$, we get that $m \equiv 0$. 
\\So we are left to consider the case when $m = z_1^+ z_1^+ z_2^+ \underbrace{m''' z_2^+ z_1^+ z_1^+}_{r}$, where $r$ is a monomial with $k(r) = k(m) - 1$ and $d(r) = l(m) - 2 - k(m) + 1 = d(m) - 1 \geq 2$. If $d(r) \geq 3$, since $k(r) < k(m)$, by the induction hypothesis, $r \equiv 0$ and so $m \equiv 0$. If $d(r) = d(m) - 1 = 2$ then $d(m) = 3$ and the number of variables in $m''$ is $l(m'') + k(m'') = l(m) - 4 + k(m) = d(m) + k(m) - 4 + k(m) = 2k(m) - 1$.
Hence there is an odd number of variables $z_i^+$, with $i \in \{ 1, 2 \}$. So, since $z_1^+ z_1^+ z z_1^+ z_1^+ \equiv 0$, as a consequence of the $*$-identities (d) and (e) of Theorem \ref{id}, we get that $m \equiv 0$ is a $*$-identity.
     
\textbf{Case 2.} $\lambda(1) \neq \emptyset$.

More in general, we shall get that any monomial $m$ in even symmetric variables and in $z_1^+, z_2^+$ with $d(m) \geq 3$ is a $*$-identity.
Consider the monomial $m(y^+, z_1^+, z_2^+) = z_1^+ y^+ z_2^+$ and let $b_1, b_2 \in (M_{1, 2}(F))^+_1$ and $a_1 = \beta e_{11} + \gamma(e_{22} + e_{33}) \in (M_{1, 2}(F))^+_0$.
It is easily checked that $m(a_1, b_1, b_2) = m'(a_1', b_1, b_2)$, where $m'(y^+, z_1^+, z_2^+) = z_1^+ z_2^+ y^+$ and $a_1' = \gamma e_{11} + \beta (e_{22} + e_{33})$.
Hence $m \equiv 0$ if and only if $m' \equiv 0$. As a consequence, we also get that $z_1^+ y_1^+ \dots y_n^+ z_2^+ \equiv 0$ if and only if $z_1^+ z_2^+ y_1^+ \dots y_n^+ \equiv 0$.
By using this property, in order to establish whether a monomial in variables $y^+$ and $z^+$ is a $*$-identity, we may assume that it is of the type $y_{i_1}^+ \dots y_{i_{k_1}}^+ z_{j_1}^+ \dots z_{j_r}^+ y_{l_1}^+ \dots y_{l_{k_2}}^+$, with $k_1, k_2, r \geq 0$.
Now, let $m = m(y_1^+, y_2^+, z_1^+, z_2^+)$ with $d(m) \geq 3$, then $m$ is a $*$-identity by the previous case and the proof is completed.
\end{proof}

By Theorem \ref{th:id}, we get the following result.

\begin{proposition}\label{prop:14}
If $\langle \lambda \rangle = (\lambda(1), \emptyset, \emptyset, \lambda(4))$ in (\ref{eqn:ultimococarattere}), with $\lambda(4) = (\rho_1 + \rho_2, \rho_2) \neq \emptyset$, then $m_{\langle \lambda \rangle} \neq 0$ if and only if $\rho_1 \leq 2$.
\end{proposition}

Next we consider the case $\lambda(3) \neq \emptyset$ and $\lambda(4) \neq \emptyset$.

\begin{proposition}\label{prop:134}
If $\langle \lambda \rangle = (\lambda(1), \emptyset, \lambda(3), \lambda(4))$ in (\ref{eqn:ultimococarattere}), with $\lambda(3) = (w_1 + w_2, w_2) \neq \emptyset$, $\lambda(4) = (\rho_1 + \rho_2, \rho_2) \neq \emptyset$ and $|w_1 - \rho_1| \leq 2$, then $m_{\langle \lambda \rangle} \neq 0$ .
\end{proposition}

\begin{proof}
Consider the elements $R_1 = e_{11}$, $R_2 = e_{22} + e_{33}$, $N_1 = e_{12} - e_{31}$, $N_2 = e_{13} + e_{21}$, $P_1 = e_{13} - e_{21}$, $P_2 = e_{12} + e_{31}$.
First suppose that $w_1 \geq \rho_1$ and distinguish some cases.
 \begin{itemize}
    \item $w_1 \leq 2$.
    \\We consider the multitableau
          $$T_{(\lambda(1), \emptyset, \lambda(3), \lambda(4))} = (T_{\lambda(1)}, \emptyset, T_{\lambda(3)}, \bar{T}_{\lambda(4)}),$$
     where $\bar{T}_{\lambda(4)}$ is the initial standard tableau on the integers $2 \alpha_2 + \alpha_2 + 2 w_2 + w_1 + 1$, $\dots$, $n$ and $T_{\lambda(1)}$ and $T_{\lambda(3)}$ are the tableaux we constructed in the proof of Proposition \ref{prop:13} on the integers $1$, $\dots$, $2 \alpha_2 + \alpha_2 + 2 w_2 + w_1$.
     Hence  
         $$f_{T_{(\lambda(1), \emptyset, \lambda(3), \lambda(4))}}(y_1^+, y_2^+, z_1^+, z_2^+, z_1^-, z_2^-) =  f_{T_{(\lambda(1), \emptyset, \lambda(3), \emptyset)}}(y_1^+, y_2^+, z_1^+, z_2^+) f_{\bar{T}_{\lambda(4)}}(z_1^-, z_2^-),$$
    where $f_{T_{(\lambda(1), \emptyset, \lambda(3), \emptyset)}}(y_1^+, y_2^+, z_1^+, z_2^+)$ are the highest weight vectors obtained in Proposition \ref{prop:13}.
    Then, if $w_1 = 0$,
        $$f_{T_{(\lambda(1), \emptyset, \lambda(3), \lambda(4))}}(R_2, R_1, N_1, N_2, P_1, P_2) = (\pm \beta e_{11} \pm e_{22} \pm e_{33}) (2^{\rho_2} e_{11} \pm e_{22} \pm e_{33}) \neq 0;$$
    if $w_1 = 1$,
        $$f_{T_{(\lambda(1), \emptyset, \lambda(3), \lambda(4))}}(R_2, R_1, N_1, N_2, P_1, P_2) = (\pm e_{31} \pm \beta e_{12}) (2^{\rho_2} e_{11} \pm e_{22} \pm e_{33}) (e_{13} - e_{21})^{\rho_1} \neq 0,$$
    where $\beta \in \{ 0, 2^{w_2} \}$;
    \\if $w_1 = 2$,
        $$f_{T_{(\lambda(1), \emptyset, \lambda(3), \lambda(4))}}(R_2, R_1, N_1, N_2, P_1, P_2) = \pm e_{32} (2^{\rho_2} e_{11} \pm e_{22} \pm e_{33}) (e_{13} - e_{21})^{\rho_1} \neq 0.$$
    
    \item $w_1 \geq 3$.
     \\Let $\lambda(1) = \emptyset$. First, let $w_1 - \rho_1 = 0$ and consider the following multitableau 
            $$T_{(\emptyset, \emptyset, \lambda(3), \lambda(4))} = (\emptyset, \emptyset, T_{\lambda(3)}, T_{\lambda(4)}),$$
        where
         \begin{center}
            $T_{\lambda(3)} =
            \begin{array}{ccccccccc}
             \hline \multicolumn{1}{|c|}{1} & \multicolumn{1}{|c|}{\ \dots \ } & \multicolumn{1}{|c|}{\substack{2 w_2 - 1}} & \multicolumn{1}{|c|}{\substack{2 w_2 + 1}} & \multicolumn{1}{|c|}{\substack{2 w_2 + 2}} & \multicolumn{1}{|c|}{\substack{2 w_2 + 2 \rho_2 + 4}} & \multicolumn{1}{|c|}{\substack{2 w_2 + 2 \rho_2 + 6}} & \multicolumn{1}{|c|}{\dots} & \multicolumn{1}{|c|}{\substack{n - 2}}\\
             \hline \multicolumn{1}{|c|}{2} & \multicolumn{1}{|c|}{\ \dots \ } & \multicolumn{1}{|c|}{\substack{2 w_2}}\\
             \cline{1-3}
            \end{array}$
         \end{center}

        and
                
         \begin{center}
            $T_{\lambda(4)} = 
            \begin{array}{ccccccccc}
             \hline \multicolumn{1}{|c|}{\substack{2 w_2 + 3}} & \multicolumn{1}{|c|}{\ \dots \ } & \multicolumn{1}{|c|}{\substack{2 w_2 + 2 \rho_2 + 1}} & \multicolumn{1}{|c|}{\substack{2 w_2 + 2 \rho_2 + 3}} & \multicolumn{1}{|c|}{\substack{2 w_2 + 2 \rho_2 + 5}} & \multicolumn{1}{|c|}{\ \dots \ } & \multicolumn{1}{|c|}{\substack{n - 3}} & \multicolumn{1}{|c|}{\substack{n - 1}} & \multicolumn{1}{|c|}{\substack{n}}\\
             \hline \multicolumn{1}{|c|}{\substack{2 w_2 + 4}} & \multicolumn{1}{|c|}{ \ \dots \ } & \multicolumn{1}{|c|}{\substack{2 w_2 + 2 \rho_2 + 2}} \\
             \cline{1-3}
            \end{array} \ $.
         \end{center}
Then
         $$f_{T_{(\emptyset, \emptyset, \lambda(3), \lambda(4))}}(z_1^+, z_2^+, z_1^-, z_2^-) = [z_1^+, z_2^+]^{w_2}(z_1^+)^2 [z_1^-, z_2^-]^{\rho_2}(z_1^- z_1^+)^{\rho_1 - 2} (z_1^-)^2$$
and
         $$f_{T_{(\emptyset, \emptyset, \lambda(3), \lambda(4))}}(N_1, N_2, P_1, P_2) = \pm e_{33} \neq 0.$$

    If $w_1 - \rho_1 = 1$, then we consider the following multitableau 
            $$T_{(\emptyset, \emptyset, \lambda(3), \lambda(4))} = (\emptyset, \emptyset, T_{\lambda(3)}, T_{\lambda(4)}),$$
        where
         \begin{center}
            $T_{\lambda(3)} =
            \begin{array}{cccccccccc}
             \hline \multicolumn{1}{|c|}{1} & \multicolumn{1}{|c|}{\ \dots \ } & \multicolumn{1}{|c|}{\substack{2 w_2 - 1}} & \multicolumn{1}{|c|}{\substack{2 w_2 + 1}} & \multicolumn{1}{|c|}{\substack{2 w_2 + 2}} & \multicolumn{1}{|c|}{\substack{2 w_2 + 2 \rho_2 + 4}} & \multicolumn{1}{|c|}{\substack{2 w_2 + 2 \rho_2 + 6}} & \multicolumn{1}{|c|}{\dots} & \multicolumn{1}{|c|}{\substack{n - 3}} & \multicolumn{1}{|c|}{\substack{n}}\\
             \hline \multicolumn{1}{|c|}{2} & \multicolumn{1}{|c|}{\ \dots \ } & \multicolumn{1}{|c|}{\substack{2 w_2}}\\
             \cline{1-3}
            \end{array}$
         \end{center}

and
     
         \begin{center}
            $T_{\lambda(4)} = 
            \begin{array}{ccccccccc}
             \hline \multicolumn{1}{|c|}{\substack{2 w_2 + 3}} & \multicolumn{1}{|c|}{\ \dots \ } & \multicolumn{1}{|c|}{\substack{2 w_2 + 2 \rho_2 + 1}} & \multicolumn{1}{|c|}{\substack{2 w_2 + 2 \rho_2 + 3}} & \multicolumn{1}{|c|}{\substack{2 w_2 + 2 \rho_2 + 5}} & \multicolumn{1}{|c|}{\ \dots \ } & \multicolumn{1}{|c|}{\substack{n - 4}} & \multicolumn{1}{|c|}{\substack{n - 2}} & \multicolumn{1}{|c|}{\substack{n - 1}}\\
             \hline \multicolumn{1}{|c|}{\substack{2 w_2 + 4}} & \multicolumn{1}{|c|}{ \ \dots \ } & \multicolumn{1}{|c|}{\substack{2 w_2 + 2 \rho_2 + 2}} \\
             \cline{1-3}
            \end{array} \ $.
         \end{center}
Then
            $$f_{T_{(\emptyset, \emptyset, \lambda(3), \lambda(4))}}(z_1^+, z_2^+, z_1^-, z_2^-) = [z_1^+, z_2^+]^{w_2}(z_1^+)^2 [z_1^-, z_2^-]^{\rho_2}(z_1^- z_1^+)^{\rho_1 - 2} (z_1^-)^2 z_1^+$$
and
            $$f_{T_{(\emptyset, \emptyset, \lambda(3), \lambda(4))}}(N_1, N_2, P_1, P_2) = \pm e_{31} \neq 0.$$

        Finally, if $w_1 - \rho_1 = 2$, then we consider the following multitableau 
            $$T_{(\emptyset, \emptyset, \lambda(3), \lambda(4))} = (\emptyset, \emptyset, T_{\lambda(3)}, T_{\lambda(4)}),$$
        where
         \begin{center}
            $T_{\lambda(3)} =
            \begin{array}{cccccccccc}
             \hline \multicolumn{1}{|c|}{1} & \multicolumn{1}{|c|}{\ \dots \ } & \multicolumn{1}{|c|}{\substack{2 w_2 - 1}} & \multicolumn{1}{|c|}{\substack{2 w_2 + 1}} & \multicolumn{1}{|c|}{\substack{2 w_2 + 2}} & \multicolumn{1}{|c|}{\substack{2 w_2 + 2 \rho_2 + 4}} & \multicolumn{1}{|c|}{\substack{2 w_2 + 2 \rho_2 + 6}} & \multicolumn{1}{|c|}{\dots} & \multicolumn{1}{|c|}{\substack{n - 2}} & \multicolumn{1}{|c|}{\substack{n}}\\
             \hline \multicolumn{1}{|c|}{2} & \multicolumn{1}{|c|}{\ \dots \ } & \multicolumn{1}{|c|}{\substack{2 w_2}}\\
             \cline{1-3}
            \end{array}$
         \end{center}

and
                
         \begin{center}
            $T_{\lambda(4)} = 
            \begin{array}{cccccccc}
             \hline \multicolumn{1}{|c|}{\substack{2 w_2 + 3}} & \multicolumn{1}{|c|}{\ \dots \ } & \multicolumn{1}{|c|}{\substack{2 w_2 + 2 \rho_2 + 1}} & \multicolumn{1}{|c|}{\substack{2 w_2 + 2 \rho_2 + 3}} & \multicolumn{1}{|c|}{\substack{2 w_2 + 2 \rho_2 + 5}} & \multicolumn{1}{|c|}{\ \dots \ } & \multicolumn{1}{|c|}{\substack{n - 3}} & \multicolumn{1}{|c|}{\substack{n - 1}}\\
             \hline \multicolumn{1}{|c|}{\substack{2 w_2 + 4}} & \multicolumn{1}{|c|}{ \ \dots \ } & \multicolumn{1}{|c|}{\substack{2 w_2 + 2 \rho_2 + 2}} \\
             \cline{1-3}
            \end{array} \ $.
         \end{center}
Then
         $$f_{T_{(\emptyset, \emptyset, \lambda(3), \lambda(4))}}(z_1^+, z_2^+, z_1^-, z_2^-) = [z_1^+, z_2^+]^{w_2}(z_1^+)^2 [z_1^-, z_2^-]^{\rho_2}(z_1^- z_1^+)^{\rho_1}$$
and
         $$f_{T_{(\emptyset, \emptyset, \lambda(3), \lambda(4))}}(N_1, N_2, P_1, P_2) = \pm e_{32} \neq 0.$$

\indent Now, let $\lambda(1) \neq \emptyset$. First suppose $w_1 - \rho_1 = 0$ and consider the multitableau $T_{(\lambda(1), \emptyset, \lambda(3), \lambda(4))}$ such that 
            $$f_{T_{(\lambda(1), \emptyset, \lambda(3), \lambda(4))}}(y_1^+, y_2^+, z_1^+, z_2^+, z_1^-, z_2^-) =$$ $$\underbrace{\tilde{y}_1^+ \dots \tilde{\tilde{y}}_1^+}_{\alpha_2} (y_1^+)^{\alpha_1} [z_1^+, z_2^+]^{w_2} (z_1^+)^2 [z_1^-, z_2^-]^{\rho_2}(z_1^- z_1^+)^{\rho_1 - 2} z_1^- \underbrace{\tilde{y}_2^+ \dots \tilde{\tilde{y}}_2^+}_{\alpha_2} z_1^-.$$
        Then
            $$f_{T_{(\lambda(1), \emptyset, \lambda(3), \lambda(4))}}(R_2, R_1, N_1, N_2, P_1, P_2) = \pm e_{33} \neq 0.$$
            
        If $w_1 - \rho_1 = 1$, then we consider the following multitableau 
            $$T_{(\lambda(1), \emptyset, \lambda(3), \lambda(4))} = (T_{\lambda(1)}, \emptyset, T_{\lambda(3)}, T_{\lambda(4)}),$$
        where
            $$T_{\lambda(1)} =
            \begin{array}{ccccc}
             \hline \multicolumn{1}{|c|}{\substack{1}} & \multicolumn{1}{|c|}{\ \dots \ } & \multicolumn{1}{|c|}{\substack{\alpha_2}} & \multicolumn{1}{|c|}{\ \dots \ } & \multicolumn{1}{|c|}{\substack{\alpha_1 + \alpha_2}}\\
             \hline \multicolumn{1}{|c|}{\substack{n - \alpha_2 + 1}} & \multicolumn{1}{|c|}{ \ \dots \ } & \multicolumn{1}{|c|}{\substack{n}} \\
             \cline{1-3}
            \end{array}$$
 
        and $T_{\lambda(3)}$ and $T_{\lambda(4)}$ are the tableaux we considered in the previous case on the integers $\alpha_1 + \alpha_2 + 1$, $\dots$, $n - \alpha_2$. Hence
            $$f_{T_{(\lambda(1), \emptyset, \lambda(3), \lambda(4))}}(y_1^+, y_2^+, z_1^+, z_2^+, z_1^-, z_2^-) = \underbrace{\tilde{y}_1^+ \dots \tilde{\tilde{y}}_1^+}_{\alpha_2} (y_1^+)^{\alpha_1} f_{T_{(\emptyset, \emptyset, \lambda(3), \lambda(4))}}(z_1^+, z_2^+, z_1^-, z_2^-) \underbrace{\tilde{y}_2^+ \dots \tilde{\tilde{y}}_2^+}_{\alpha_2}$$
        where $f_{T_{(\emptyset, \emptyset, \lambda(3), \lambda(4))}}(z_1^+, z_2^+, z_1^-, z_2^-)$ is the highest weight vector obtained in the previous case.
        Then
            $$f_{T_{(\lambda(1), \emptyset, \lambda(3), \lambda(4))}}(R_2, R_1, N_1, N_2, P_1, P_2) = \pm e_{31} \neq 0.$$
     
        Finally, if $w_1 - \rho_1 = 2$, then we consider the multitableau $T_{\langle \lambda \rangle}$ such that
            $$f_{T_{(\lambda(1), \emptyset, \lambda(3), \lambda(4))}}(y_1^+, y_2^+, z_1^+, z_2^+, z_1^-, z_2^-) =$$ $$\underbrace{\tilde{y}_1^+ \dots \tilde{\tilde{y}}_1^+}_{\alpha_2} (y_1^+)^{\alpha_1} [z_1^+, z_2^+]^{w_2} (z_1^+)^2 [z_1^-, z_2^-]^{\rho_2}(z_1^- z_1^+)^{\rho_1 - 1} z_1^- \underbrace{\tilde{y}_2^+ \dots \tilde{\tilde{y}}_2^+}_{\alpha_2} z_1^+.$$
        Then
            $$f_{T_{(\lambda(1), \emptyset, \lambda(3), \lambda(4))}}(R_2, R_1, N_1, N_2, P_1, P_2) = \pm e_{32} \neq 0.$$
 \end{itemize}

By Theorem \ref{th:id}, the same approach applies in case $\rho_1 \geq w_1$. Now the proof is completed.
\end{proof}

Now we can ask: is it true that if the multiplicity is different than zero then $|w_1 - \rho_1| \leq 2$? We consider an easy case in which $|w_1 - \rho_1| \geq 3$.

\begin{example}
    Let $\langle \lambda \rangle = (\emptyset, \emptyset, (1, 1), (3))$ in (\ref{eqn:ultimococarattere}). Notice that $|w_1 - \rho_1| = 4$. It is easy to see that for any $T_{(\emptyset, \emptyset, (1, 1), (3))}$ the corresponding highest weight vector $f_{T_{(\emptyset, \emptyset, (1, 1), (3))}}$ is a $*$-identity being a consequence of the set of $*$-identities
     $$\big\{ z_1^+ z_2^+ z_3^+ + z_2^+ z_3^+ z_1^+ + z_3^+ z_1^+ z_2^+, \ \ \ \tilde{z}_1^+ y^- \tilde{z}_2^+, \ \ \ \tilde{z}_1^+ z^- \tilde{z}_2^+ z^- z^-, \ \ \ z^- \tilde{z}_1^+ z^- \tilde{z}_2^+ z^-, \ \ \ z^- z^- \tilde{z}_1^+ z^- \tilde{z}_2^+ \big\}.$$
    Then $m_{\langle \lambda \rangle} = 0$ and the same result holds if $\langle \lambda \rangle = (\emptyset, \emptyset, (3), (1, 1))$, by Theorem \ref{th:id}.
\end{example}

Motivated by this example we are led to think that the following conjecture holds.

\begin{conjecture}\label{conj:1}
    Let $(\lambda(1), \emptyset, \lambda(3), \lambda(4))$ in (\ref{eqn:ultimococarattere}), with $\lambda(3) = (w_1 + w_2, w_2) \neq \emptyset$, $\lambda(4) = (\rho_1 + \rho_2, \rho_2) \neq \emptyset$. Then $|w_1 - \rho_1| \leq 2$ if and only if $m_{\langle \lambda \rangle} \neq 0$.
\end{conjecture}

Given a real number $c$, let $\lceil c \rceil$ denote the ceiling of $c$, i.e., the smallest integer greater than or equal to $c$.

Now we consider the case $\lambda(2) \neq \emptyset$ and $\lambda(3) \neq \emptyset$.

\begin{proposition}\label{prop:123}
Let $\langle \lambda \rangle = (\lambda(1), \lambda(2), \lambda(3), \emptyset)$ in (\ref{eqn:ultimococarattere}), with $\lambda(1) = (\alpha_1 + \alpha_2, \alpha_2)$, $\lambda(2) = (\gamma_1 + \gamma_2 + \gamma_3, \gamma_2 + \gamma_3, \gamma_3) \neq \emptyset$, $\lambda(3) = (w_1 + w_2, w_2) \neq \emptyset$. If 
 \begin{itemize}
 
    \item $w_1 \leq 2$ or
    
    \item $w_1 \geq 3$ and $\gamma_1 + \gamma_2 \geq \lceil \frac{w_1}{2} \rceil - 1$
    
 \end{itemize} 
then $m_{\langle \lambda \rangle} \neq 0$.
\end{proposition}

\begin{proof}
Consider the elements $R_1 = e_{11}$, $R_2 = e_{22} + e_{33}$, $M_1 = e_{23} + e_{32}$, $M_2 = e_{23} - e_{32}$, $M_3 = e_{22} - e_{33}$, $N_1 = e_{12} - e_{31}$, \ $N_2 = e_{13} + e_{21}$.
First, suppose $w_1 \leq 2$ and we distinguish two cases.
\\Let $\lambda(1) = \emptyset$ and consider the initial standard multitableau. Hence  
        $$f_{\bar{T}_{(\emptyset, \lambda(2), \lambda(3), \emptyset)}}(y_1^-, y_2^-, y_3^-, z_1^+, z_2^+) = f_{\bar{T}_{\lambda(2)}}(y_1^-, y_2^-, y_3^-) f_{\bar{T}_{\lambda(3)}}(z_1^+, z_2^+)$$
    and
     $$f_{\bar{T}_{(\emptyset, \lambda(2), \lambda(3), \emptyset)}}(M_1, M_2, M_3, N_1, N_2) = 
         \begin{cases}
          r (
          \pm e_{22} \pm e_{33})
          (e_{12} - e_{31})^{w_1} \neq 0, & \text{if } \gamma_1 \text{ is even}, \\
          r 
          (\pm e_{23} \pm e_{32})
          (e_{12} - e_{31})^{w_1}  \neq 0, & \text{if } \gamma_1 \text{ is odd},
         \end{cases}$$
    where $r = \pm 6^{\gamma_3} 2^{\gamma_2}$.
  \\Then, let $\lambda(1) \neq \emptyset$. We consider the multitableau
      $$T_{(\lambda(1), \lambda(2), \lambda(3), \emptyset)} = (T_{\lambda(1)}, \bar{T}_{\lambda(2)}, T_{\lambda(3)}, \emptyset),$$
  where $\bar{T}_{\lambda(2)}$ is the initial standard tableau on the integers $1$, $\dots$, $3 \gamma_3 + 2 \gamma_2 + \gamma_1$ and $T_{\lambda(1)}$ and $T_{\lambda(3)}$ are the tableaux we constructed in the proof of Proposition \ref{prop:13} on the integers $3 \gamma_3 + 2 \gamma_2 + \gamma_1 + 1$, $\dots$, $n$.
  Hence  
  \begin{center}
     $f_{T_{(\lambda(1), \lambda(2), \lambda(3), \emptyset)}}(y_1^+, y_2^+, y_1^-, y_2^-, y_3^-, z_1^+, z_2^+) = f_{\bar{T}_{\lambda(2)}}(y_1^-, y_2^-, y_3^-) f_{T_{(\lambda(1), \emptyset, \lambda(3), \emptyset)}}(y_1^+, y_2^+, z_1^+, z_2^+)$,
  \end{center}
 where $f_{T_{(\lambda(1), \emptyset, \lambda(3), \emptyset)}}(y_1^+, y_2^+, z_1^+, z_2^+)$ are the highest weight vectors obtained in Proposition \ref{prop:13}.
 \\If $w_1 = 0$, then 
     $$f_{T_{(\lambda(1), \lambda(2), \lambda(3), \emptyset)}}(R_2, R_1, M_1, M_2, M_3, N_1, N_2) = 
     \begin{cases} 
      r (e_{22} \pm e_{33}) \neq 0, & \text{if } \gamma_1 \text{ is even}, \\
      r (e_{23} \pm e_{32}) \neq 0, & \text{if } \gamma_1 \text{ is odd},      
     \end{cases}$$
 if $w_1 = 1$, then
     $$f_{T_{(\lambda(1), \lambda(2), \lambda(3), \emptyset)}}(R_2, R_1, M_1, M_2, M_3, N_1, N_2) = 
     \begin{cases} 
      r e_{31} \neq 0, & \text{if } \gamma_1 \text{ is even}, \\
      r e_{21} \neq 0, & \text{if } \gamma_1 \text{ is odd},      
     \end{cases}$$
 if $w_1 = 2$, then
     $$f_{T_{(\lambda(1), \lambda(2), \lambda(3), \emptyset)}}(R_2, R_1, M_1, M_2, M_3, N_1, N_2) = 
     \begin{cases} 
      r e_{32} \neq 0, & \text{if } \gamma_1 \text{ is even}, \\
      r e_{22} \neq 0, & \text{if } \gamma_1 \text{ is odd},
     \end{cases}$$
 where $r = \pm 6^{\gamma_3} 2^{\gamma_2}$.

 Now, suppose $w_1 \geq 3$ and $\gamma_1 + \gamma_2 \geq k - 1$, with $\lceil \frac{w_1}{2} \rceil = k$. We consider the Young diagram $D_{\lambda(2)} = D_\mu | D_\nu$, corresponding to the partition $\lambda(2)$, as obtained by gluing two diagrams $D_\mu$ and $D_\nu$, where $D_\mu$ is the subdiagram of $D_{\lambda(2)}$ made up of $\gamma_1 + \gamma_2 + \gamma_3 - (k - 1)$ columns and $D_\nu$ is the diagram made up of the last $k - 1$ columns.
 We also consider the Young diagram $D_{\lambda(3)} = D_\beta | D_\delta$ obtained by gluing the diagram $D_\beta$ that has $w_2$ columns of height two and two columns of height one and the diagram $D_\delta$ that has $w_1 - 2$ columns of height one. We define $t_i$ as the number of columns of height $i$ in $D_\mu$, with $t_i \leq \gamma_i$, for $i \in \{ 1, 2 \}$, such that $t_1 + t_2 = \gamma_1 + \gamma_2 - (k + 1)$.
 We distinguish the following cases.
    \begin{itemize}
 
    \item $w_1 = 2k$.
      \\If $\gamma_1 \geq k - 1$, we consider the multitableau $T_{(\lambda(1), \lambda(2), \lambda(3), \emptyset)}$ such that
            $$f_{T_{(\lambda(1), \lambda(2), \lambda(3), \emptyset)}}(y_1^+, y_2^+, y_1^-, y_2^-, y_3^-, z_1^+, z_2^+) =$$
            $$\underbrace{\tilde{y}_1^+ \dots \tilde{\tilde{y}}_1^+}_{\alpha_2}(y_1^+)^{\alpha_1} f_{\bar{T}_{\mu}}(y_1^-, y_2^-, y_3^-) f_{\bar{T}_{\beta}}(z_1^+, z_2^+) (y_1^- (z_1^+)^2)^{k - 2} y_1^- z_1^+ \underbrace{\tilde{y}_2^+ \dots \tilde{\tilde{y}}_2^+}_{\alpha_2} z_1^+.$$
        Then
            $$f_{T_{(\lambda(1), \lambda(2), \lambda(3), \emptyset)}}(R_2, R_1, M_1, M_2, M_3, N_1, N_2) = 
            \begin{cases} 
             r e_{32} \neq 0, & \text{if } \gamma_1 - k + 1 \text{ is even}, \\
             r e_{22} \neq 0, & \text{if } \gamma_1 - k + 1 \text{ is odd},
            \end{cases}$$
        where $r = \pm 6^{\gamma_3} 2^{\gamma_2}$.

        If $0 \leq \gamma_1 < k - 1$, we consider the multitableau $T_{(\lambda(1), \lambda(2), \lambda(3), \emptyset)}$ such that
            $$f_{T_{(\lambda(1), \lambda(2), \lambda(3), \emptyset)}}(y_1^+, y_2^+, y_1^-, y_2^-, y_3^-, z_1^+, z_2^+) =$$
            $$\underbrace{\tilde{y}_1^+ \dots \tilde{\tilde{y}}_1^+}_{\alpha_2}(y_1^+)^{\alpha_1} f_{\bar{T}_{\mu}}(y_1^-, y_2^-, y_3^-) f_{\bar{T}_{\beta}}(z_1^+, z_2^+) ([y_1^-, y_2^-] (z_1^+)^2)^{k - 2 - \gamma_1} [y_1^-, y_2^-] z_1^+ \underbrace{\tilde{y}_2^+ \dots \tilde{\tilde{y}}_2^+}_{\alpha_2} z_1^+ (y_1^- (z_1^+)^2)^{\gamma_1}.$$
        Then
            $$f_{T_{(\lambda(1), \lambda(2), \lambda(3), \emptyset)}}(R_2, R_1, M_1 + M_3, M_2, M_3, N_1, N_2) = 
            \begin{cases} 
             \pm 6^{\gamma_3} 2^{\frac{1}{2} (s + 1)} e_{32} \neq 0, & \text{if } t_1 + t_2 \text{ is even}, \\
             \pm 6^{\gamma_3} 2^{\frac{1}{2} s} (e_{22} + e_{32}) \neq 0, & \text{if } t_1 + t_2 \text{ is odd},
            \end{cases}$$
        where $s = 3 \gamma_2 - k + \gamma_1$.
        
    \item $w_1 = 2k - 1$.
    \\We consider the multitableau $T_{(\lambda(1), \lambda(2), \lambda(3), \emptyset)}$ such that, if $\gamma_1 \neq 0$,
        $$f_{T_{(\lambda(1), \lambda(2), \lambda(3), \emptyset)}}(y_1^+, y_2^+, y_1^-, y_2^-, y_3^-, z_1^+, z_2^+) =$$
        $$\underbrace{\tilde{y}_1^+ \dots \tilde{\tilde{y}}_1^+}_{\alpha_2} (y_1^+)^{\alpha_1}   f_{\bar{T}_{\mu}}(y_1^-, y_2^-, y_3^-) f_{\bar{T}_{\beta}}(z_1^+, z_2^+) ([y_1^-, y_2^-] (z_1^+)^2)^{\gamma_2 - t_2} (y_1^- (z_1^+)^2)^{\gamma_1 - t_1 - 1} y_1^- z_1^+ \underbrace{\tilde{y}_2^+ \dots \tilde{\tilde{y}}_2^+}_{\alpha_2}$$
    and, if $\gamma_1 = 0$,
            $$f_{T_{(\lambda(1), \lambda(2), \lambda(3), \emptyset)}}(y_1^+, y_2^+, y_1^-, y_2^-, y_3^-, z_1^+, z_2^+) =$$
            $$\underbrace{\tilde{y}_1^+ \dots \tilde{\tilde{y}}_1^+}_{\alpha_2} (y_1^+)^{\alpha_1} f_{\bar{T}_{\mu}}(y_1^-, y_2^-, y_3^-) f_{\bar{T}_{\beta}}(z_1^+, z_2^+) ([y_1^-, y_2^-] (z_1^+)^2)^{\gamma_2 - t_2 - 1} [y_1^-, y_2^-] z_1^+ \underbrace{\tilde{y}_2^+ \dots \tilde{\tilde{y}}_2^+}_{\alpha_2}.$$

    So, if $\gamma_1 \geq k - 1$,
        $$f_{T_{(\lambda(1), \lambda(2), \lambda(3), \emptyset)}}(R_2, R_1, M_1, M_2, M_3, N_1, N_2) = 
        \begin{cases} 
         r e_{31} \neq 0, & \text{if } \gamma_1 - k + 1 \text{ is even}, \\
         r e_{21} \neq 0, & \text{if } \gamma_1 - k + 1 \text{ is odd}      
        \end{cases}$$
    and, if $0 \leq \gamma_1 < k - 1$,
        $$f_{T_{(\lambda(1), \lambda(2), \lambda(3), \emptyset)}}(R_2, R_1, M_1 + M_3, M_2, M_3, N_1, N_2) = 
        \begin{cases} 
         \pm 6^{\gamma_3} 2^{\frac{1}{2} (s + 1)} e_{31} \neq 0, & \text{if } t_1 + t_2 \text{ is even}, \\
         \pm 6^{\gamma_3} 2^{\frac{1}{2} s} (e_{21} + e_{31}) \neq 0, & \text{if } t_1 + t_2 \text{ is odd},
        \end{cases}$$
    where $r = \pm 6^{\gamma_3} 2^{\gamma_2}$ and $s = 3 \gamma_2 - k + \gamma_1$.
\end{itemize}\end{proof}

Consider the following example.

\begin{example}
    Let $\langle \lambda \rangle = (\emptyset, \lambda(2), (w_1), \emptyset)$, with $\lambda(2) = (\gamma_1 + \gamma_2 + \gamma_3, \gamma_2 + \gamma_3, \gamma_3) \vdash s$, $w_1 \geq 5$ and $s < \lceil \frac{w_1}{2} \rceil - 1$. Then $m_{\langle \lambda \rangle} = 0$, since, for any $T_{\langle \lambda \rangle}$, each monomial of $f_{T_{\langle \lambda \rangle}}$ contains at least three consecutive odd symmetric variables and, so, it is a $*$-identity (see Theorem \ref{id} (b)). 
    \\Notice that, if $\langle \lambda \rangle = (\emptyset, (1, 1, 1), (3), \emptyset)$ we also get that $m_{\langle \lambda \rangle} = 0$, since any highest weight vector $f_{T_{\langle \lambda \rangle}}$ is a consequence of the following set of $*$-identities 
         $$\big\{ [y^+, y], \ \ \ z_1^+ z_2^+ z_3^+ + z_2^+ z_3^+ z_1^+ + z_3^+ z_1^+ z_2^+, \ \ \ y_1^- z^+ y_2^-, \ \ \ \tilde{y}^-_1 \tilde{y}^-_2 z^+ z^+ \tilde{y}^-_3 z^+,$$
         $$\tilde{y}^-_1 z^+ z^+ \tilde{y}^-_2 \tilde{y}^-_3 z^+, \ \ \ z^+ \tilde{y}^-_1 \tilde{y}^-_2 z^+ z^+ \tilde{y}^-_3, \ \ \ z^+ \tilde{y}^-_1 z^+ z^+ \tilde{y}^-_2 \tilde{y}^-_3 \big\}.$$ 
\end{example}

Motivated by the above example, we make the following conjecture.

\begin{conjecture}\label{conj:2}
Let $\langle \lambda \rangle = (\lambda(1), \lambda(2), \lambda(3), \emptyset)$ in (\ref{eqn:ultimococarattere}), with $\lambda(1) = (\alpha_1 + \alpha_2, \alpha_2)$, $\lambda(2) = (\gamma_1 + \gamma_2 + \gamma_3, \gamma_2 + \gamma_3, \gamma_3) \neq \emptyset$, $\lambda(3) = (w_1 + w_2, w_2) \neq \emptyset$. Then
 \begin{itemize}

    \item $w_1 \leq 2$ or
    
    \item $w_1 \geq 3$ and $\gamma_1 + \gamma_2 \geq \lceil \frac{w_1}{2} \rceil - 1$
    
 \end{itemize} 
if and only if $m_{\langle \lambda \rangle} \neq 0$.
\end{conjecture}

According to Theorem \ref{th:id}, if $\lambda(2) \neq \emptyset$ and $\lambda(4) \neq \emptyset$, we obtain the following result.

\begin{proposition}\label{prop:124}
If $\langle \lambda \rangle = (\lambda(1), \lambda(2), \emptyset, \lambda(4))$ in (\ref{eqn:ultimococarattere}), with $\lambda(2) = (\gamma_1 + \gamma_2 + \gamma_3, \gamma_2 + \gamma_3, \gamma_3) \neq \emptyset$ $\lambda(4) = (\rho_1 + \rho_2, \rho_2) \neq \emptyset$ and 
 \begin{itemize}

    \item $\rho_1 \leq 2$ or
    
    \item $\rho_1 \geq 3$ and $\gamma_1 + \gamma_2 \geq \lceil \frac{\rho_1}{2} \rceil - 1$,
    
 \end{itemize}
then $m_{\langle \lambda \rangle} \neq 0$.
\end{proposition}

Now we consider the case $\lambda(2) \neq \emptyset$, $\lambda(3) \neq \emptyset$ and $\lambda(4) \neq \emptyset$.

\begin{proposition}\label{prop:234}
Let $\langle \lambda \rangle = (\emptyset, \lambda(2), \lambda(3), \lambda(4))$ in (\ref{eqn:ultimococarattere}), with $\lambda(2) = (\gamma_1 + \gamma_2 + \gamma_3, \gamma_2 + \gamma_3, \gamma_3) \neq \emptyset$, $\lambda(3) = (w_1 + w_2, w_2) \neq \emptyset$, $\lambda(4) = (\rho_1 + \rho_2, \rho_2) \neq \emptyset$ and let $l = max\{w_1, \rho_1\}$. If 
 \begin{itemize}
     \item $|w_1 - \rho_1| \leq 2$ or
     \item $|w_1 - \rho_1| \geq 3$ and $\gamma_1 + \gamma_2 \geq \lceil \frac{l}{2} \rceil - 1$,
  \end{itemize}
then $m_{\langle \lambda \rangle} \neq 0$.
\end{proposition}

\begin{proof}
Consider the elements $M_1 = e_{23} + e_{32}$, $M_2 = e_{23} - e_{32}$, $M_3 = e_{22} - e_{33}$, $N_1 = e_{12} - e_{31}$, $N_2 = e_{13} + e_{21}$, $P_1 = e_{13} - e_{21}$, $P_2 = e_{12} + e_{31}$.
Because of Theorem \ref{th:id}, we may assume $w_1 \geq \rho_1$.
First suppose $w_1 - \rho_1 \leq 2$ and we distinguish two cases.
 \begin{itemize}
 
    \item $w_1 \leq 2$.
        \\If $\rho_1 \in \{ 0, 1 \}$, we consider the initial standard multitableau, then
            $$f_{T_{(\emptyset, \lambda(2), \lambda(3), \lambda(4))}}(y_1^-, y_2^-, y_3^-, z_1^+, z_2^+, z_1^-, z_2^-) = f_{\bar{T}_{\lambda(2)}}(y_1^-, y_2^-, y_3^-) f_{\bar{T}_{\lambda(3)}}(z_1^+, z_2^+) f_{\bar{T}_{\lambda(4)}} (z_1^-, z_2^-).$$
        So,
            $$f_{T_{(\emptyset, \lambda(2), \lambda(3), \lambda(4))}}(M_1, M_2, M_3, N_1, N_2, P_1, P_2) =$$ 
            $$\begin{cases} 
             \pm 6^{\gamma_3} 2^{\gamma_2} (\pm e_{22} \pm e_{33}) (e_{12} - e_{31})^{w_1} (2^{\rho_2} e_{11} \pm e_{22} \pm e_{33} ) (e_{13} - e_{21})^{\rho_1} \neq 0, & \text{if } \gamma_1 \text{ is even}, \\
             \pm 6^{\gamma_3} 2^{\gamma_2} (\pm e_{23} \pm e_{32}) (e_{12} - e_{31})^{w_1} (2^{\rho_2} e_{11} \pm e_{22} \pm e_{33} ) (e_{13} - e_{21})^{\rho_1} \neq 0, & \text{if } \gamma_1 \text{ is odd}.
            \end{cases}$$

        If $\rho_1 = 2$, we consider the following multitableau
            $$T_{(\emptyset, \lambda(2), \lambda(3), \lambda(4))} = (\emptyset, \bar{T}_{\lambda (2)}, \bar{T}_{\lambda (3)}, \bar{T}_{\lambda (4)}),$$
        where $\bar{T}_{\lambda (2)}$ is the initial standard tableau on the integers $2 w_2 + w_1 + 2 \rho_2 + \rho_1 + 1$, $\dots$, $n$, $\bar{T}_{\lambda (3)}$ is the initial standard tableau on the integers $1$, $\dots$, $2 w_2 + w_1$ and $\bar{T}_{\lambda (4)}$ is the initial standard tableau on the integers $2 w_2 + w_1 + 1$, $\dots$, $2 w_2 + w_1 + 2 \rho_2 + \rho_1$.
        Then 
            $$f_{T_{(\emptyset, \lambda(2), \lambda(3), \lambda(4))}}(y_1^-, y_2^-, y_3^-, z_1^+, z_2^+, z_1^-, z_2^-) =  f_{\bar{T}_{\lambda(3)}}(z_1^+, z_2^+) f_{\bar{T}_{\lambda(4)}} (z_1^-, z_2^-) f_{\bar{T}_{\lambda(2)}}(y_1^-, y_2^-, y_3^-)$$
        and so,
            $$f_{T_{(\emptyset, \lambda(2), \lambda(3), \lambda(4))}}(M_1, M_2, M_3, N_1, N_2, P_1, P_2) = 
            \begin{cases} 
              r e_{33} \neq 0, & \text{if } \gamma_1 \text{ is even}, \\
             r e_{32} \neq 0, & \text{if } \gamma_1 \text{ is odd},
            \end{cases}$$
        with $r = \pm 6^{\gamma_3} 2^{\gamma_2}$.

    \item $w_1 \geq 3$. 
    \\We consider the multitableau
        $$T_{(\emptyset, \lambda(2), \lambda(3), \lambda(4))} = (\emptyset, \bar{T}_{\lambda(2)}, T_{\lambda(3)}, T_{\lambda(4)}),$$
    where $\bar{T}_{\lambda(2)}$ is the initial standard tableau on the integers $1$, $\dots$, $3 \gamma_3 + 2 \gamma_2 + \gamma_1$ and $T_{\lambda(3)}$ and $T_{\lambda(4)}$ are the tableaux we considered in the proof of Proposition \ref{prop:134} on the integers $3 \gamma_3 + 2 \gamma_2 + \gamma_1 + 1$, $\dots$, $n$.
    Then 
        $$f_{T_{(\emptyset, \lambda(2), \lambda(3), \lambda(4))}} = f_{\bar{T}_{\lambda(2)}}(y_1^-, y_2^-, y_3^-) f_{T_{(\emptyset, \emptyset, \lambda(3), \lambda(4))}}(z_1^+, z_2^+, z_1^-, z_2^-),$$
    where $f_{T_{(\emptyset, \emptyset, \lambda(3), \lambda(4))}}(z_1^+, z_2^+, z_1^-, z_2^-)$ is the highest weight vector obtained in Proposition \ref{prop:134}.
    \\So, if $\rho_1 = w_1 - 2$,
        $$f_{T_{(\emptyset, \lambda(2), \lambda(3), \lambda(4))}}(M_1, M_2, M_3, N_1, N_2, P_1, P_2) = 
        \begin{cases} 
         r e_{32} \neq 0, & \text{if } \gamma_1 \text{ is even}, \\
         r e_{22} \neq 0, & \text{if } \gamma_1 \text{ is odd};
        \end{cases}$$
    if $\rho_1 = w_1 - 1$,
        $$f_{T_{(\emptyset, \lambda(2), \lambda(3), \lambda(4))}}(M_1, M_2, M_3, N_1, N_2, P_1, P_2) = 
        \begin{cases} 
         r e_{31} \neq 0, & \text{if } \gamma_1 \text{ is even}, \\
         r e_{21} \neq 0, & \text{if } \gamma_1 \text{ is odd};
        \end{cases}$$
    if $\rho_1 = w_1$,
        $$f_{T_{(\emptyset, \lambda(2), \lambda(3), \lambda(4))}}(M_1, M_2, M_3, N_1, N_2, P_1, P_2) = 
        \begin{cases} 
         r e_{33} \neq 0, & \text{if } \gamma_1 \text{ is even}, \\
         r e_{23} \neq 0, & \text{if } \gamma_1 \text{ is odd},
        \end{cases}$$
    with $r = \pm 6^{\gamma_3} 2^{\gamma_2}$.
   
    \end{itemize}

Now, consider the case $w_1 - \rho_1 \geq 3$, so $l = w_1$ and $\gamma_1 + \gamma_2 \geq k - 1$, with $\lceil \frac{l}{2} \rceil = k$. We consider the Young diagrams $D_{\lambda(2)} = D_\mu | D_\nu$ and $D_{\lambda(3)} = D_\beta | D_\delta$ as in the proof of Proposition \ref{prop:123} and $D_{\lambda(4)} = D_\epsilon | D_\sigma$, such that $D_\epsilon$ is the subdiagram of $D_{\lambda(4)}$ made up of the first $\rho_2$ columns and $D_\sigma$ is the diagram obtained considering the last $\rho_1$ columns.
We define the integer $a$ as
           $$a := 
           \begin{cases} 
            \rho_1, & \text{if } \rho_1 \text{ is even}, \\
            \rho_1 - 1, & \text{if } \rho_1 \text{ is odd}.
           \end{cases}$$
We distinguish these different cases.
 \begin{itemize}
       
    \item $w_1 = 2k$.
    \\We consider the multitableau $T_{\langle \lambda \rangle}$
       such that: 
       \\if $\gamma_1 \geq k - 1$,
            $$f_{T_{(\emptyset, \lambda(2), \lambda(3), \lambda(4))}}(y_1^-, y_2^-, y_3^-, z_1^+, z_2^+, z_1^-, z_2^-) = g_1(y_1^-, y_2^-, y_3^-, z_1^+, z_2^+, z_1^-, z_2^-) (y_1^- (z_1^+)^2)^{k - 1 - \frac{a}{2}} (z_1^-)^{\rho_1 - a},$$ 
    where
           $$g_1(y_1^-, y_2^-, y_3^-, z_1^+, z_2^+, z_1^-, z_2^-) = f_{\bar{T}_{\mu}}(y_1^-, y_2^-, y_3^-) f_{\bar{T}_{\beta}}(z_1^+, z_2^+) f_{\bar{T}_{\epsilon}} (z_1^-, z_2^-) (y_1^- (z_1^+)^2 (z_1^-)^2)^{\frac{a}{2}};$$
      if $0 < \gamma_1 < k - 1$ and $\frac{a}{2} \leq k - 1 - \gamma_1$,
            $$f_{T_{(\lambda(1), \lambda(2), \lambda(3), \lambda(4))}}(y_1^-, y_2^-, y_3^-, z_1^+, z_2^+, z_1^-, z_2^-) = g_2(y_1^-, y_2^-, y_3^-, z_1^+, z_2^+, z_1^-, z_2^-) (y_1^- (z_1^+)^2)^{\gamma_1} (z_1^-)^{\rho_1 - a},$$
    where
           $$g_2(y_1^-, y_2^-, y_3^-, z_1^+, z_2^+, z_1^-, z_2^-) =$$
           $$f_{\bar{T}_{\mu}}(y_1^-, y_2^-, y_3^-) f_{\bar{T}_{\beta}}(z_1^+, z_2^+) f_{\bar{T}_{\epsilon}} (z_1^-, z_2^-) ([y_1^-, y_2^-] (z_1^+)^2 (z_1^-)^2)^{\frac{a}{2}} ([y_1^-, y_2^-] (z_1^+)^2)^{k - 1 - \gamma_1 - \frac{a}{2}};$$
      if $0 < \gamma_1 < k - 1$ and $\frac{a}{2} > k - 1 - \gamma_1$,
            $$f_{T_{(\emptyset, \lambda(2), \lambda(3), \lambda(4))}}(y_1^-, y_2^-, y_3^-, z_1^+, z_2^+, z_1^-, z_2^-) = g_3(y_1^-, y_2^-, y_3^-, z_1^+, z_2^+, z_1^-, z_2^-) (y_1^- (z_1^+)^2)^{k - 1 - \frac{a}{2}} (z_1^-)^{\rho_1 - a},$$
    where
           $$g_3(y_1^-, y_2^-, y_3^-, z_1^+, z_2^+, z_1^-, z_2^-) =$$
           $$f_{\bar{T}_{\mu}}(y_1^-, y_2^-, y_3^-) f_{\bar{T}_{\beta}}(z_1^+, z_2^+) f_{\bar{T}_{\epsilon}} (z_1^-, z_2^-) ([y_1^-, y_2^-] (z_1^+)^2 (z_1^-)^2)^{k - 1 - \gamma_1} (y_1^- (z_1^+)^2 (z_1^-)^2)^{\frac{a}{2} - k + 1 + \gamma_1};$$
      if $\gamma_1 = 0$,
            $$f_{T_{(\emptyset, \lambda(2), \lambda(3), \lambda(4))}}(y_1^-, y_2^-, y_3^-, z_1^+, z_2^+, z_1^-, z_2^-) = g_4(y_1^-, y_2^-, y_3^-, z_1^+, z_2^+, z_1^-, z_2^-) ([y_1^-, y_2^-] (z_1^+)^2)^{k - 1 - \frac{a}{2}} (z_1^-)^{\rho_1 - a},$$
    where
           $$g_4(y_1^-, y_2^-, y_3^-, z_1^+, z_2^+, z_1^-, z_2^-) = f_{\bar{T}_{\mu}}(y_1^-, y_2^-, y_3^-) f_{\bar{T}_{\beta}}(z_1^+, z_2^+) f_{\bar{T}_{\epsilon}} (z_1^-, z_2^-) ([y_1^-, y_2^-] (z_1^+)^2 (z_1^-)^2)^{\frac{a}{2}}.$$
      
       So, if $\gamma_1 \geq k - 1$,
            $$f_{T_{(\emptyset, \lambda(2), \lambda(3), \lambda(4))}}(M_1, M_2, M_3, N_1, N_2, P_1 + P_2, P_2) = 
            \begin{cases} 
             r e_{32} (- e_{21})^{\rho_1 - a} \neq 0, & \text{if } \gamma_1 - k + 1 \text{ is even}, \\
             r e_{22} (- e_{21})^{\rho_1 - a} \neq 0, & \text{if } \gamma_1 - k + 1 \text{ is odd},
            \end{cases}$$
       where $r = \pm 6^{\gamma_3} 2^{\gamma_2}$.
       \\And, if $0 \leq \gamma_1 < k - 1$,
           $$f_{T_{(\emptyset, \lambda(2), \lambda(3), \lambda(4))}}(M_1 + M_3, M_2, M_3, N_1, N_2, P_1, P_2) =
           \begin{cases} 
            \pm 6^{\gamma_3} 2^{\frac{1}{2} (s + 1)} e_{32} (- e_{21})^{\rho_1 - a} \neq 0, & \text{if } t_1 + t_2 \text{ is even}, \\
            \pm 6^{\gamma_3} 2^{\frac{1}{2} s} (e_{22} + e_{32}) (- e_{21})^{\rho_1 - a} \neq 0, & \text{if } t_1 + t_2 \text{ is odd},
           \end{cases}$$
       where $s = 3 \gamma_2 - k + \gamma_1$.

       \item $w_1 = 2k - 1$.
       \\We consider the multitableau $T_{(\emptyset, \lambda(2), \lambda(3), \lambda(4))}$ such that: 
       \\if $\gamma_1 \geq k - 1$,
            $$f_{T_{(\emptyset, \lambda(2), \lambda(3), \lambda(4))}}(y_1^-, y_2^-, y_3^-, z_1^+, z_2^+, z_1^-, z_2^-) =
            g_1(y_1^-, y_2^-, y_3^-, z_1^+, z_2^+, z_1^-, z_2^-) (y_1^- (z_1^+)^2)^{k - 2 - \frac{a}{2}} y_1^- z_1^+ (z_1^-)^{\rho_1 - a};$$
      if $0 < \gamma_1 < k - 1$ and $\frac{a}{2} \leq k - 1 - \gamma_1$,
            $$f_{T_{(\emptyset, \lambda(2), \lambda(3), \lambda(4))}}(y_1^-, y_2^-, y_3^-, z_1^+, z_2^+, z_1^-, z_2^-) =
            g_2(y_1^-, y_2^-, y_3^-, z_1^+, z_2^+, z_1^-, z_2^-)(y_1^- (z_1^+)^2)^{\gamma_1 - 1} y_1^- z_1^+ (z_1^-)^{\rho_1 - a};$$
      if $0 < \gamma_1 < k - 1$ and $\frac{a}{2} > k - 1 - \gamma_1$,
            $$f_{T_{(\emptyset, \lambda(2), \lambda(3), \lambda(4))}}(y_1^-, y_2^-, y_3^-, z_1^+, z_2^+, z_1^-, z_2^-) =
            g_3(y_1^-, y_2^-, y_3^-, z_1^+, z_2^+, z_1^-, z_2^-)(y_1^- (z_1^+)^2)^{k - 2 - \frac{a}{2}} y_1^- z_1^+ (z_1^-)^{\rho_1 - a};$$
      if $\gamma_1 = 0$
            $$f_{T_{(\emptyset, \lambda(2), \lambda(3), \lambda(4))}}(y_1^-, y_2^-, y_3^-, z_1^+, z_2^+, z_1^-, z_2^-) =$$
            $$g_4(y_1^-, y_2^-, y_3^-, z_1^+, z_2^+, z_1^-, z_2^-)([y_1^-, y_2^-] (z_1^+)^2)^{k - 2 - \frac{a}{2}} [y_1^-, y_2^-] z_1^+ (z_1^-)^{\rho_1 - a}.$$
      
       Then, if $\gamma_1 \geq k - 1$,
           $$f_{T_{(\emptyset, \lambda(2), \lambda(3), \lambda(4))}}(M_1, M_2, M_3, N_1, N_2, P_1 + P_2, P_2) = 
           \begin{cases} 
            r e_{31} (e_{13} + e_{12})^{\rho_1 - a} \neq 0, & \text{if } t_1 + t_2 \text{ is even}, \\
            r e_{21} (e_{13} + e_{12})^{\rho_1 - a} \neq 0, & \text{if } t_1 + t_2 \text{ is odd},      
           \end{cases}$$
       where $r = \pm 6^{\gamma_3} 2^{\gamma_2}$.
       \\And, if $0 \leq \gamma_1 < k - 1$,
           $$f_{T_{(\emptyset, \lambda(2), \lambda(3), \lambda(4))}}(M_1 + M_3, M_2, M_3, N_1, N_2, P_1, P_2) =
           \begin{cases} 
            \pm 6^{\gamma_3} 2^{\frac{1}{2} (s + 1)} e_{31} (e_{13})^{\rho_1 - a}  \neq 0, & \text{if } t_1 + t_2 \text{ is even}, \\
            \pm 6^{\gamma_3} 2^{\frac{1}{2} s} (e_{21} + e_{31}) (e_{13})^{\rho_1 - a} \neq 0, & \text{if } t_1 + t_2 \text{ is odd},
           \end{cases}$$
       where $s = 3 \gamma_2 - k + \gamma_1$.

    \end{itemize}
\end{proof}

Now we are ready to prove the main theorem of this section.

\begin{theorem}\label{th:final}
Consider 
 $$\chi_n^{*}(M_{1,2}(F)) = \sum_{\substack{\langle \lambda \rangle \vdash n, \\ h(\lambda(1)) \leq 2, \ h(\lambda(2)) \leq 3, \\ h(\lambda(3)) \leq 2, \ h(\lambda(4)) \leq 2}} m_{\langle \lambda \rangle} \chi_ {\langle \lambda \rangle}$$ the $n$-th $*$-cocharacter of $M_{1, 2}(F)$.
Let $\langle \lambda \rangle = (\lambda(1), \lambda(2), \lambda(3), \lambda(4))$, with $\lambda(2) = (\gamma_1 + \gamma_2 + \gamma_3, \gamma_2 + \gamma_3, \gamma_3)$, $\lambda(3) = (w_1 + w_2, w_2)$, $\lambda(4) = (\rho_1 + \rho_2, \rho_2)$ and $l = max \{w_1, \rho_1 \}$. The following statements hold.

If $\lambda(3) = \lambda(4) = \emptyset$, then $m_{\langle \lambda \rangle} \neq 0$ if and only if $h(\lambda(1)) \leq 1$.

If $m_{\langle \lambda \rangle} \neq 0$, $\lambda(2) = \emptyset$, $\lambda(i) = \emptyset$ and $\lambda(j) \neq \emptyset$, with $i, j \in \{ 3, 4 \}$, $i \neq j$, then $|w_1 - \rho_1| \leq 2$.

If $\lambda(j) \neq \emptyset$ for some $j \in \{ 3, 4 \}$ and
 \begin{enumerate}
     \item[$\cdot$] $|w_1 - \rho_1| \leq 2$ or 
     \item[$\cdot$] $|w_1 - \rho_1| \geq 3$ and $\lambda(2) \neq \emptyset$, with $\gamma_1 + \gamma_2 \geq \lceil \frac{l}{2} \rceil - 1$
 \end{enumerate}

then $m_{\langle \lambda \rangle} \neq 0$.
\end{theorem}

\begin{proof}
The first two items follow from Propositions \ref{prop:12}, \ref{prop:13} and \ref{prop:14}.

Let $\lambda(j) \neq \emptyset$ for some $j \in \{ 3, 4 \}$. 
If $\lambda(2) = \emptyset$ and $|w_1 - \rho_1| \leq 2$, then $m_{\langle \lambda \rangle} \neq 0$ by Propositions \ref{prop:13}, \ref{prop:14} and \ref{prop:134}. 
We may assume, since the previous cases hold, that $\lambda(2) \neq \emptyset$. 
If $\lambda(i) = \emptyset$, for $i \in \{ 3, 4 \}, \ i \neq j$, $|w_1 - \rho_1| \leq 2$ or $|w_1 - \rho_1| \geq 3$ and $\gamma_1 + \gamma_2 \geq \lceil \frac{l}{2} \rceil - 1$, then we get the result by Propositions \ref{prop:123} and \ref{prop:124}.
Now, let $\lambda(i) \neq \emptyset$, for $i \in \{ 3, 4 \}, \ i \neq j$.
If $\lambda(1) = \emptyset$, $|w_1 - \rho_1| \leq 2$ or $|w_1 - \rho_1| \geq 3$ and $\gamma_1 + \gamma_2 \geq \lceil \frac{l}{2} \rceil - 1$, then $m_{\langle \lambda \rangle} \neq 0$ by Proposition \ref{prop:234}.

We are left to prove the case $\lambda(i) \neq \emptyset$ for all $i \in \{ 1, 2, 3, 4 \}$, $|w_1 - \rho_1| \leq 2$ or $|w_1 - \rho_1| \geq 3$ and $\gamma_1 + \gamma_2 \geq \lceil \frac{l}{2} \rceil - 1$.
Consider the elements $R_1 = e_{11}$, $R_2 = e_{22} + e_{33}$, $M_1 = e_{23} + e_{32}$, $M_2 = e_{23} - e_{32}$, $M_3 = e_{22} - e_{33}$, $N_1 = e_{12} - e_{31}$, $N_2 = e_{13} + e_{21}$, $P_1 = e_{13} - e_{21}$, $P_2 = e_{12} + e_{31}$.
 
Because of Theorem \ref{th:id}, we may suppose $w_1 \geq \rho_1$.

If $w_1 - \rho_1 \leq 2$, we distinguish two cases:
 \begin{itemize}

    \item $w_1 \leq 2$.
         \\If $\rho_1 \in \{ 0, 1 \}$, we consider the multitableau
         \begin{center}
            $T_{(\lambda(1), \lambda(2), \lambda(3), \lambda(4))} = (T_{\lambda(1)}, \bar{T}_{\lambda(2)}, T_{\lambda(3)}, T_{\lambda(4)})$,
         \end{center}
        where $\bar{T}_{\lambda(2)}$ is the initial standard tableau on the integers $1$, $\dots$, $3 \gamma_3 + 2 \gamma_2 + \gamma_1$ and $T_{\lambda(1)}$, $T_{\lambda(3)}$ and $T_{\lambda(4)}$ are the tableaux we considered in the proof of Proposition \ref{prop:134}. 
        Then
            $$f_{T_{(\lambda(1), \lambda(2), \lambda(3), \lambda(4))}}(y_1^+, y_2^+, y_1^-, y_2^-, y_3^-, z_1^+, z_2^+, z_1^-, z_2^-) =$$ 
            $$f_{\bar{T}_{\lambda(2)}}(y_1^-, y_2^-, y_3^-) f_{T_{(\lambda(1), \emptyset, \lambda(3), \lambda(4))}}(y_1^+, y_2^+, z_1^+, z_2^+, z_1^-, z_2^-),$$
        where $f_{T_{(\lambda(1), \emptyset, \lambda(3), \lambda(4))}}(y_1^+, y_2^+, z_1^+, z_2^+, z_1^-, z_2^-)$ is the highest weight vector obtained in Proposition \ref{prop:134}.
        So,
            $$f_{T_{(\lambda(1), \lambda(2), \lambda(3), \lambda(4))}}(R_2, R_1, M_1, M_2, M_3, N_1, N_2, P_1, P_2) =$$ 
            $$\begin{cases} 
             \pm 6^{\gamma_3} 2^{\gamma_2} (\pm e_{22} \pm e_{33}) (e_{12} - e_{31})^{w_1} (2^{\rho_2} e_{11} \pm e_{22} \pm e_{33}) (e_{13} - e_{21})^{\rho_1}) \neq 0, & \text{if } \gamma_1 \text{ is even}, \\
             \pm 6^{\gamma_3} 2^{\gamma_2} (\pm e_{23} \pm e_{32}) (e_{12} - e_{31})^{w_1} (2^{\rho_2} e_{11} \pm e_{22} \pm e_{33} ) (e_{13} - e_{21})^{\rho_1}) \neq 0, & \text{if } \gamma_1 \text{ is odd}.
            \end{cases}$$

        If $\rho_1 = 2$, we consider the following multitableau
            $$T_{(\lambda(1), \lambda(2), \lambda(3), \lambda(4))} = (T_{\lambda(1)}, \bar{T}_{\lambda (2)}, T_{\lambda (3)}, T_{\lambda (4)}),$$
        where $\bar{T}_{\lambda(2)}$ is the initial standard tableau on the integers $n - 3 \gamma_3 - 2 \gamma_2 - \gamma_1 + 1$, $\dots$, $n$ and $T_{\lambda(1)}$, $T_{\lambda(3)}$ and $T_{\lambda(4)}$ are the tableaux we considered in the proof of Proposition \ref{prop:134} on the integers $1$, $\dots$, $n - 3 \gamma_3 - 2 \gamma_2 - \gamma_1$.
        \\Then 
            $$f_{T_{(\lambda(1), \lambda(2), \lambda(3), \lambda(4))}}(y_1^+, y_2^+, y_1^-, y_2^-, y_3^-, z_1^+, z_2^+, z_1^-, z_2^-) =$$
            $$f_{T_{(\lambda(1), \emptyset, \lambda(3), \lambda(4))}}(y_1^+, y_2^+, z_1^+, z_2^+, z_1^-, z_2^-) f_{\bar{T}_{\lambda(2)}}(y_1^-, y_2^-, y_3^-),$$
        where $f_{T_{(\lambda(1), \emptyset, \lambda(3), \lambda(4))}}(y_1^+, y_2^+, z_1^+, z_2^+, z_1^-, z_2^-)$ is the highest weight vector obtained in Proposition \ref{prop:134}.
        \\Then,
            $$f_{T_{(\lambda(1), \lambda(2), \lambda(3), \lambda(4))}}(R_2, R_1, M_1, M_2, M_3, N_1, N_2, P_1, P_2) = 
            \begin{cases} 
             r e_{33} \neq 0, & \text{if } \gamma_1 \text{ is even}, \\
             r e_{32} \neq 0, & \text{if } \gamma_1 \text{ is odd},
            \end{cases}$$
        with $r = \pm 6^{\gamma_3} 2^{\gamma_2}$.

    \item $w_1 \geq 3$. 
    \\We consider the multitableau
        $$T_{(\lambda(1), \lambda(2), \lambda(3), \lambda(4))} = (T_{\lambda(1)}, \bar{T}_{\lambda(2)}, T_{\lambda(3)}, T_{\lambda(4)}),$$
    where $\bar{T}_{\lambda(2)}$ is the initial standard tableau on the integers $1$, $\dots$, $3 \gamma_3 + 2 \gamma_2 + \gamma_1$ and $T_{\lambda(1)}$, $T_{\lambda(3)}$ and $T_{\lambda(4)}$ are the tableaux we considered in the proof of Proposition \ref{prop:134} on the integers $3 \gamma_3 + 2 \gamma_2 + \gamma_1 + 1$, $\dots$, $n$.
    \\Then 
        $$f_{T_{(\lambda(1), \lambda(2), \lambda(3), \lambda(4))}} = f_{\bar{T}_{\lambda(2)}}(y_1^-, y_2^-, y_3^-) f_{T_{(\lambda(1), \emptyset, \lambda(3), \lambda(4))}}(y_1^+, y_2^+, z_1^+, z_2^+, z_1^-, z_2^-),$$
    where $f_{T_{(\lambda(1), \emptyset, \lambda(3), \lambda(4))}}(y_1^+, y_2^+, z_1^+, z_2^+, z_1^-, z_2^-)$ is the highest weight vector obtained in Proposition \ref{prop:134}.
    \\If $\rho_1 = w_1 - 2$,
        $$f_{T_{(\lambda(1), \lambda(2), \lambda(3), \lambda(4))}}(R_2, R_1, M_1, M_2, M_3, N_1, N_2, P_1, P_2) = 
        \begin{cases} 
         r e_{32} \neq 0, & \text{if } \gamma_1 \text{ is even}, \\
         r e_{22} \neq 0, & \text{if } \gamma_1 \text{ is odd};
        \end{cases}$$
    if $\rho_1 = w_1 - 1$,
        $$f_{T_{(\lambda(1), \lambda(2), \lambda(3), \lambda(4))}}(M_1, M_2, M_3, N_1, N_2, P_1, P_2) = 
        \begin{cases} 
         r e_{31} \neq 0, & \text{if } \gamma_1 \text{ is even}, \\
         r e_{21} \neq 0, & \text{if } \gamma_1 \text{ is odd};
        \end{cases}$$
    if $\rho_1 = w_1$,
        $$f_{T_{(\lambda(1), \lambda(2), \lambda(3), \lambda(4))}}(M_1, M_2, M_3, N_1, N_2, P_1, P_2) = 
        \begin{cases} 
         r e_{33} \neq 0, & \text{if } \gamma_1 \text{ is even}, \\
         r e_{23} \neq 0, & \text{if } \gamma_1 \text{ is odd},
        \end{cases}$$
    with $r = \pm 6^{\gamma_3} 2^{\gamma_2}$.
    
 \end{itemize} 

Now, if $w_1 - \rho_1 \geq 3$ and $\gamma_1 + \gamma_2 \geq k - 1$, where $k = \lceil \frac{w_1}{2} \rceil$, we define the integer $a$ as in Proposition \ref{prop:234},
 \begin{center}
           $a := 
           \begin{cases} 
            \rho_1, & \text{if } \rho_1 \text{ is even}, \\
            \rho_1 - 1, & \text{if } \rho_1 \text{ is odd}
           \end{cases}$
        \end{center}
and we consider the $*$-polynomials $g_i (y_1^-, y_2^-, y_3^-, z_1^+, z_2^+, z_1^-, z_2^-)$ constructed in Proposition \ref{prop:234}.
\\We distinguish these different cases.
 \begin{itemize}
       
    \item $w_1 = 2k$.
    \\We consider the multitableau $T_{(\lambda(1), \lambda(2), \lambda(3), \lambda(4))}$
       such that: 
       \\if $\gamma_1 \geq k - 1$,
            $$f_{T_{(\lambda(1), \lambda(2), \lambda(3), \lambda(4))}}(y_1^+, y_2^+, y_1^-, y_2^-, y_3^-, z_1^+, z_2^+, z_1^-, z_2^-) =$$
            $$\underbrace{\tilde{y}_1^+ \dots \tilde{\tilde{y}}_1^+}_{\alpha_2} (y_1^+)^{\alpha_1} g_1(y_1^-, y_2^-, y_3^-, z_1^+, z_2^+, z_1^-, z_2^-) (y_1^- (z_1^+)^2)^{k - 2 - \frac{a}{2}} y_1^- z_1^+ \underbrace{\tilde{y}_2^+ \dots \tilde{\tilde{y}}_2^+}_{\alpha_2} z_1^+ (z_1^-)^{\rho_1 - a};$$

      if $0 < \gamma_1 < k - 1$ and $\frac{a}{2} \leq k - 1 - \gamma_1$,
            $$f_{T_{(\lambda(1), \lambda(2), \lambda(3), \lambda(4))}}(y_1^+, y_2^+, y_1^-, y_2^-, y_3^-, z_1^+, z_2^+, z_1^-, z_2^-) =$$
            $$\underbrace{\tilde{y}_1^+ \dots \tilde{\tilde{y}}_1^+}_{\alpha_2} (y_1^+)^{\alpha_1} g_2(y_1^-, y_2^-, y_3^-, z_1^+, z_2^+, z_1^-, z_2^-) (y_1^- (z_1^+)^2)^{\gamma_1 - 1} y_1^- z_1^+ \underbrace{\tilde{y}_2^+ \dots \tilde{\tilde{y}}_2^+}_{\alpha_2} z_1^+ (z_1^-)^{\rho_1 - a};$$
      if $0 < \gamma_1 < k - 1$ and $\frac{a}{2} > k - 1 - \gamma_1$,
            $$f_{T_{(\lambda(1), \lambda(2), \lambda(3), \lambda(4))}}(y_1^+, y_1^+, y_1^-, y_2^-, y_3^-, z_1^+, z_2^+, z_1^-, z_2^-) =$$
            $$\underbrace{\tilde{y}_1^+ \dots \tilde{\tilde{y}}_1^+}_{\alpha_2} (y_1^+)^{\alpha_1} g_3(y_1^-, y_2^-, y_3^-, z_1^+, z_2^+, z_1^-, z_2^-) (y_1^- (z_1^+)^2)^{k - 2 - \frac{a}{2}} y_1^- z_1^+ \underbrace{\tilde{y}_2^+ \dots \tilde{\tilde{y}}_2^+}_{\alpha_2} z_1^+ (z_1^-)^{\rho_1 - a}$$
      and if $\gamma_1 = 0$
            $$f_{T_{(\lambda(1), \lambda(2), \lambda(3), \lambda(4))}}(y_1^+, y_2^+, y_1^-, y_2^-, y_3^-, z_1^+, z_2^+, z_1^-, z_2^-) =$$
            $$\underbrace{\tilde{y}_1^+ \dots \tilde{\tilde{y}}_1^+}_{\alpha_2} (y_1^+)^{\alpha_1} g_4(y_1^-, y_2^-, y_3^-, z_1^+, z_2^+, z_1^-, z_2^-) ([y_1^-, y_2^-] (z_1^+)^2)^{k - 2 - \frac{a}{2}} [y_1^-, y_2^-] z_1^+ \underbrace{\tilde{y}_2^+ \dots \tilde{\tilde{y}}_2^+}_{\alpha_2} z_1^+ (z_1^-)^{\rho_1 - a}.$$

       So, if $\gamma_1 \geq k - 1$,
            $$f_{T_{(\lambda(1), \lambda(2), \lambda(3), \lambda(4))}}(R_2, R_1, M_1, M_2, M_3, N_1, N_2, P_1 + P_2, P_2) = \begin{cases} 
             r e_{32} (- e_{21})^{\rho_1 - a} \neq 0, & \text{if } \gamma_1 - k + 1 \text{ is even}, \\
             r e_{22} (- e_{21})^{\rho_1 - a} \neq 0, & \text{if } \gamma_1 - k + 1 \text{ is odd},
            \end{cases}$$
       where $r = \pm 6^{\gamma_3} 2^{\gamma_2}$
       \\and, if $0 \leq \gamma_1 < k - 1$,
           $$f_{T_{(\lambda(1), \lambda(2), \lambda(3), \lambda(4))}}(R_2, R_1, M_1 + M_3, M_2, M_3, N_1, N_2, P_1, P_2) = \begin{cases} 
            \pm 6^{\gamma_3} 2^{\frac{1}{2} (s + 1)} e_{32} (- e_{21})^{\rho_1 - a} \neq 0, & \text{if } t_1 + t_2 \text{ is even}, \\
            \pm 6^{\gamma_3} 2^{\frac{1}{2} s} (e_{22} + e_{32}) (- e_{21})^{\rho_1 - a} \neq 0, & \text{if } t_1 + t_2 \text{ is odd},
           \end{cases}$$
       where $s = 3 \gamma_2 - k + \gamma_1$.

       \item $w_1 = 2k - 1$. 
       \\We consider the multitableau $T_{(\lambda(1), \lambda(2), \lambda(3), \lambda(4))}$ such that
            $$f_{T_{(\lambda(1), \lambda(2), \lambda(3), \lambda(4))}}(y_1^+, y_2^+, y_1^-, y_2^-, y_3^-, z_1^+, z_2^+, z_1^-, z_2^-) =$$ $$\underbrace{\tilde{y}_1^+ \dots \tilde{\tilde{y}}_1^+}_{\alpha_2} (y_1^+)^{\alpha_1} f^\eta (y_1^+, y_2^+, y_1^-, y_2^-, y_3^-, z_1^+, z_2^+, z_1^-, z_2^-) \underbrace{\tilde{y}_2^+ \dots \tilde{\tilde{y}}_2^+}_{\alpha_2} (z_1^-)^{\rho_1 - a},$$
       where $f^\eta (y_1^+, y_2^+, y_1^-, y_2^-, y_3^-, z_1^+, z_2^+, z_1^-, z_2^-)$ is such that 
            $$f_{T_{(\emptyset, \lambda(2), \lambda(3), \lambda(4))}} = f^\eta (y_1^+, y_2^+, y_1^-, y_2^-, y_3^-, z_1^+, z_2^+, z_1^-, z_2^-) (z_1^-)^{\rho_1 - a}$$
       and $f_{T_{(\emptyset, \lambda(2), \lambda(3), \lambda(4))}}$ is the highest weight vector constructed in Proposition \ref{prop:234}.
       \\Then, if $\gamma_1 \geq k - 1$,
           $$f_{T_{(\lambda(1), \lambda(2), \lambda(3), \lambda(4))}}(R_2, R_1, M_1, M_2, M_3, N_1, N_2, P_1 + P_2, P_2) = \begin{cases} 
            \pm r e_{31} (e_{13} + e_{12})^{\rho_1 - a} \neq 0, & \text{if } \gamma_1 - k + 1 \text{ is even}, \\
            \pm r e_{21} (e_{13} + e_{12})^{\rho_1 - a} \neq 0, & \text{if } \gamma_1 - k + 1 \text{ is odd},      
           \end{cases}$$
       where $r = 6^{\gamma_3} 2^{\gamma_2}$.
       \\If $0 \leq \gamma_1 < k - 1$,
           $$f_{T_{\langle \lambda \rangle}}(R_2, R_1, M_1 + M_3, M_2, M_3, N_1, N_2, P_1, P_2) = \begin{cases} 
            \pm 6^{\gamma_3} 2^{\frac{1}{2} (s + 1)} e_{31} (e_{13})^{\rho_1 - a}  \neq 0, & \text{if } t_1 + t_2 \text{ is even}, \\
            \pm 6^{\gamma_3} 2^{\frac{1}{2} s} (e_{21} + e_{31}) (e_{13})^{\rho_1 - a} \neq 0, & \text{if } t_1 + t_2 \text{ is odd},
           \end{cases}$$
       where $s = 3 \gamma_2 - k + \gamma_1$.

    \end{itemize}

So, the proof is complete.
\end{proof}

By Conjectures \ref{conj:1}, \ref{conj:2} and Theorem \ref{th:final}, we are led to conjecture the following.

\begin{conjecture}
Consider 
 $$\chi_n^{*}(M_{1,2}(F)) = \sum_{\substack{\langle \lambda \rangle \vdash n, \\ h(\lambda(1)) \leq 2, \ h(\lambda(2)) \leq 3, \\ h(\lambda(3)) \leq 2, \ h(\lambda(4)) \leq 2}} m_{\langle \lambda \rangle} \chi_ {\langle \lambda \rangle}$$ the $n$-th $*$-cocharacter of $M_{1, 2}(F)$.
Let $\langle \lambda \rangle = (\lambda(1), \lambda(2), \lambda(3), \lambda(4))$, with $\lambda(2) = (\gamma_1 + \gamma_2 + \gamma_3, \gamma_2 + \gamma_3, \gamma_3)$, $\lambda(3) = (w_1 + w_2, w_2)$, $\lambda(4) = (\rho_1 + \rho_2, \rho_2)$ and $l = max \{w_1, \rho_1 \}$. 
Then $m_{\langle \lambda \rangle} \neq 0$ if and only if $\lambda(3) = \lambda(4) = \emptyset$ and $h(\lambda(1)) \leq 1$ or $\lambda(j) \neq \emptyset$ for some $j \in \{ 3, 4 \}$ and $|w_1 - \rho_1| \leq 2$ or $\lambda(j) \neq \emptyset$ for some $j \in \{ 3, 4 \}$, $|w_1 - \rho_1| \geq 3$ and $\lambda(2) \neq \emptyset$, with $\gamma_1 + \gamma_2 \geq \lceil \frac{l}{2} \rceil - 1$.
\end{conjecture}

\end{document}